\documentclass[12pt,oneside,leqno]{amsart}
\usepackage[margin=1in]{geometry}
\usepackage{amsmath}
\usepackage{amsthm}
\usepackage{amsfonts}
\usepackage{amssymb}
\usepackage{stmaryrd}
\usepackage[hidelinks]{hyperref}
\usepackage{tikz}
\usetikzlibrary{cd}
\usepackage{mathrsfs}

\newenvironment{nouppercase}{%
  \renewcommand{\uppercasenonmath}[1]{}}{}

\usepackage{comment}

\makeatletter
\renewcommand{\eqref}[1]{\textup{(\ignorespaces\ref{#1}\unskip\@@italiccorr)}}
\def\maketag@@@#1{\hbox{\m@th\normalfont\bfseries#1}}
\makeatother

\numberwithin{equation}{section}
\numberwithin{figure}{section}

\theoremstyle{plain}
\newtheorem{thm}[equation]{Theorem}
\newtheorem*{SepProb}{Separativity Problem}
\newtheorem*{MainThm0}{Separativity via Unit-regularity}
\newtheorem*{MainThm1}{Diagonal Reduction of Invertible Matrices}
\newtheorem*{MainThm2}{Diagonal Reduction over Separative Regular Rings}
\newtheorem{lemma}[equation]{Lemma}
\newtheorem{cor}[equation]{Corollary}
\newtheorem{prop}[equation]{Proposition}
\newtheorem{example}[equation]{Example}

\theoremstyle{definition}
\newtheorem{definition}[equation]{Definition}
\newtheorem{question}[equation]{Question}
\newtheorem{remark}[equation]{Remark}

\newcommand{\Q}{\ensuremath \mathbb{Q}}

\newcommand{\Z}{\ensuremath \mathbb{Z}}
\newcommand{\N}{\ensuremath \mathbb{N}}
\newcommand{\M}{\ensuremath \mathbb{M}}

\DeclareMathOperator{\End}{End}

\DeclareMathOperator{\im}{im}

\DeclareMathOperator{\coker}{coker}

\DeclareMathOperator{\U}{U}

\DeclareMathOperator{\FP}{FP}

\newcommand\subsetsim{\mathrel{\substack{
  \textstyle\subset\\[-0.2ex]\textstyle\sim}}}

\begin{document}

\title[Separativity, varieties, and diagonal reduction]{The separativity problem in terms of varieties and diagonal reduction}

\author{Pere Ara}
\address{Departament de Matemàtiques, Universitat Autònoma de Barcelona, 08193 Bellaterra, Barcelona, Spain, and Centre de Recerca Matemàtica, Edifici Cc, Campus de Bellaterra, 08193 Cerdanyola del Vallès, Barcelona, Spain}
\email{pere.ara@uab.cat}
\urladdr{https://www.crm.cat/person/77/ara-pere/}

\author{Ken Goodearl}
\address{Department of Mathematics, University of California, Santa Barbara, CA 93106, USA}
\email{goodearl@math.ucsb.edu}
\urladdr{https://web.math.ucsb.edu/~goodearl/}

\author{Pace P.\ Nielsen}
\address{Department of Mathematics, Brigham Young University, Provo, UT 84602, USA}
\email{pace@math.byu.edu}
\urladdr{https://mathdept.byu.edu/~pace/}

\author{Enrique Pardo}
\address{Departamento de Matemáticas, Facultad de Ciencias, Universidad de Cádiz, Campus de Puerto Real, 11510 Puerto Real (Cádiz), Spain}
\email{enrique.pardo@uca.es}
\urladdr{https://sites.google.com/gm.uca.es/webofenriquepardo/inicio}

\author{Francesc Perera}
\address{Departament de Matemàtiques, Universitat Autònoma de Barcelona, 08193 Bellaterra, Barcelona, Spain, and Centre de Recerca Matemàtica, Edifici Cc, Campus de Bellaterra, 08193 Cerdanyola del Vallès, Barcelona, Spain}
\email{francesc.perera@uab.cat}
\urladdr{https://mat.uab.cat/web/perera}

\keywords{diagonally reduce, exchange rings, inner inverses, regular rings, separativity}
\subjclass[2020]{Primary 16E50, Secondary 15A20, 16E20, 16U40, 16U60}

\begin{abstract}
We provide two new formulations of the separativity problem.

First, it is known that separativity (and strong separativity) in von Neumann regular (and exchange) rings is tightly connected to unit-regularity of certain kinds of elements.  By refining this information, we characterize separative regular rings in terms of a special type of inner inverse operation, which is defined via a single identity.  This shows that the separative regular rings form a subvariety of the regular rings.  The separativity problem reduces to the question of whether every element of the form $(1-aa')bac(1-a'a)$ in a regular ring is unit-regular, where $a'$ is an inner inverse for $a$.

Second, it is known that separativity in exchange rings is equivalent to regular matrices over corner rings being reducible to diagonal matrices via elementary row and column operations.  We show that for $2\times 2$ invertible matrices, four elementary operations are sufficient, and in general also necessary.  Dropping the invertibility hypothesis, but specializing to separative regular rings, we show that three elementary row operations together with three elementary column operations are sufficient, and again in general also necessary.  The separativity problem can subsequently be reframed in terms of the explicit number of operations needed to diagonally reduce.
\end{abstract}

\begin{nouppercase}
\maketitle
\end{nouppercase}

\section{Introduction}

In 1954, Kaplansky \cite{Kaplansky} posed three problems for abelian groups, which are:
\begin{itemize}
\item If $A\subsetsim B$ and $B\subsetsim A$ (that is, each is isomorphic to a submodule of the other), then is $A\cong B$?\smallskip
\item If $A\oplus A\cong B\oplus B$, then is $A\cong B$?\smallskip
\item If $C$ is finitely generated and $A\oplus C\cong B\oplus C$, then is $A\cong B$?
\end{itemize}
These questions have since become a standard against which to judge one's understanding of a category of modules over a ring $R$.   Depending on the qualities of the ring $R$, and on the characteristics of the module category under consideration, these questions can range from ``trivial'' to ``important unsolved problem.''

When $R$ is a (von Neumann) regular ring, and we focus on the category $\FP(R)$ of finitely generated projective right $R$-modules, it is sometimes possible to characterize when Kaplansky's problems have positive answers.  For example, the third question has a positive answer in $\FP(R)$ for exactly those regular rings that are unit-regular; this was first proven true by Handelman, see \cite[Theorem 2]{Handelman}. Unit-regular rings are also the regular rings with stable rank $1$, as shown in the independent work of Fuchs \cite[Corollary 1 and Theorem 4]{Fuchs} and Kaplansky (unpublished).  For additional related results and history see \cite[Chapter 4]{Goodearl}.

One can significantly strengthen the second hypothesis from Kaplansky's third test problem, by asking for any two modules $A,B\in \mathscr{C}$, where $\mathscr{C}$ is some category of modules:
\begin{equation}\label{Eq:SeparativityDef}
\text{If $A\oplus A\cong A\oplus B\cong B\oplus B$, then is $A\cong B$?}
\end{equation}
(This strengthens the second hypothesis of the third test question by combining the cases when $C=A$ and $C=B$.  Interestingly, this is also a specialization of the second test question.)  When the answer to \eqref{Eq:SeparativityDef} is positive, then following the literature one says that $\mathscr{C}$ is \emph{separative}.  Moreover, in the case when $\mathscr{C}=\FP(R)$, one says that the ring $R$ is \emph{separative}.

Since separativity depends on isomorphisms between modules and not equalities, then instead of working in the (large) category $\FP(R)$, it is occasionally more convenient to work with $V(R)$, the set of \emph{isomorphism types} of modules in $\FP(R)$.  The set $V(R)$ is a monoid, by defining for any two modules $A,B\in \FP(R)$ the addition
\[
[A]+[B]=[A\oplus B].
\]
Separativity of $R$ is equivalent to the assertion that $2[A]=[A]+[B]=2[B]$ implies $[A]=[B]$, in $V(R)$.

The initial study of separativity for regular rings, and more generally for exchange rings, began with the papers \cite{AGOPEarly} and \cite{AGOP}.  Since that time, many other papers have explored this topic. Separativity has a number of important and restrictive consequences in regular rings; see \cite[Theorem 6.1]{AGOP} for five examples.  For each of those five properties, it is an open question whether every regular ring has that property.  Thus, a major open problem in this area, which motivates much of this paper, is the following.

\begin{SepProb}
Is every regular \textup{(}or exchange\textup{)} ring separative?
\end{SepProb}

This problem has now been open for over 30 years.  Outside the class of exchange rings, separativity can easily fail, as for instance with the first Weyl algebra and with the coordinate ring of the 2-sphere; see \cite[Section 2]{GoodearlPaper}.  However, perhaps the most instructive parallel for us is in the category of finite rank, torsion-free abelian groups (abstractly these are just the subgroups of finite-dimensional vector spaces over the rational field $\Q$). This category is well-documented as the home of pathology in module theory (see \cite{FuchsI,FuchsII}). Despite this, examples of failure of separative cancellation are rare, and tied up with nontrivial number theory; see \cite{OMearaVinsonhaler} for details, as well as a concrete construction of a large class of explicit examples where separativity fails.

In this paper we continue developing the theory of separative regular rings.  Specifically, after a brief review of facts about inner inverses in Section \ref{Section:InnerInverses}, we explain in Section \ref{Section:NewInnerInverses} how to define separative regular (and separative exchange) rings in new ways in terms of regularity and unit-regularity.  A key result is the following ring-theoretic reformulation of separativity.

\begin{MainThm0}[Corollary {\ref{Cor:RegSepUReg} with Lemma \ref{Lemma:ShiftBackCorner}}]
A regular ring $R$, with an inner inverse operation $'$, is separative if and only if for any idempotent $e\in R$, and for any $a,b,c\in eRe$, the element
\[
x:=(1-aa')bac(1-a'a) +(1-e)
\]
is unit-regular in $R$.
\end{MainThm0}

We also prove a similar, but more complicated, characterization of separativity in exchange rings.  Moreover, the methods apply to other cancellation properties besides separativity, as multiple examples demonstrate.

A consequence of the previous theorem is that the separativity problem is equivalent to determining whether or not every such element $x$ in a regular ring is unit-regular (where we may even take $e=1$, by passing to a corner ring).  Another consequence of this work is a description of separative regular rings as a subvariety of regular rings.  This is worked out explicitly in Section \ref{Sec:FirstVariety}, while varieties of exchange rings appear in Section \ref{Sec:SecondVariety}.  A third consequence is a simpler diagonal reduction process for $2\times 2$ regular matrices over separative exchange rings.  In the last couple of sections we undertake the task of minimizing the number of elementary operations needed to diagonally reduce such matrices.  Ultimately, we prove the following two other main results:

\begin{MainThm1}[{Theorem \ref{Thm:FourIsEnough}}]
Over a separative exchange ring, every invertible $2\times 2$ matrix can be diagonally reduced using four elementary operations.
\end{MainThm1}

Generally, four elementary operations are also necessary.  Removing the invertibility hypothesis raises the number of needed operations, even when we restrict to regular rings.  In that case, we have:

\begin{MainThm2}[{Theorem \ref{Thm:SixWithThreeOnBothSides}}]
Over a separative regular ring, every $2\times 2$ matrix can be diagonally reduced using just three row operations together with three column operations. Nothing less on either side can work in general.
\end{MainThm2}

Readers may find Theorem \ref{Thm:DiagBy3Classification} and Proposition \ref{Prop:CompletableRows} of independent interest.  The first concerns connections between stable rank $1$ and diagonal reduction.  The second is a statement about completable unimodular rows in exchange rings.

\section{Inner inverses}\label{Section:InnerInverses}

Throughout this paper, let $R$ be a unital ring.  Recall that an element $a\in R$ is (von Neumann) \emph{regular} if it has an \emph{inner inverse} $b\in R$ (also called a \emph{quasi-inverse}), which satisfies $aba=a$.  Inner inverses are in general not unique; for example, any element of $R$ is an inner inverse for $0$.

When every element of a ring $R$ is regular, then $R$ is called a \emph{regular ring}.  These rings are a central fixture of modern noncommutative ring theory.  The reader is directed to \cite{Goodearl} for a more thorough study of these rings.  In case $R$ is regular, then we may as well fix some inner inverse $a'\in R$ for each element $a\in R$.  We can then view $'$ as a unary operation on $R$, called an \emph{inner inverse operation}.  The category of regular rings with a distinguished inner inverse operation is a variety, in the universal algebra sense, satisfying the axioms of rings together with the identity
\[
aa'a=a.
\]

In general, inner inverse operations on a given ring are not unique.  In fact, there are many potential ways to improve a given inner inverse.  For example, let $(R,\, ')$ be a regular ring (with a distinguished inner inverse operation), and let $a\in R$.  Fixing
\[
a^{\dagger}:=a'aa',
\]
then a quick computation shows that
\[
aa^{\dagger}a=a\ \text{ and }\ a^{\dagger}aa^{\dagger}=a^{\dagger}.
\]
Thus, $a^{\dagger}$ is an inner inverse for $a$, and in addition $a$ is an inner inverse for $a^{\dagger}$.  In this case, we say that $a$ and $a^{\dagger}$ form a pair of \emph{reflexive inverses}.  We can view $^{\dagger}$ as a unary operation on $R$, called a \emph{reflexive inverse operation}.

Since every regular ring has a reflexive inverse operation, it is tempting to simply declare that the initial operation $'$ is a reflexive inverse operation, and indeed one can do so.  But why stop there?  In fact, there is an infinite sequence of further refinements that can be made to inner inverse operations, as shown in \cite{BergmanStrongInner,NS,OMearaBetterInverses}.  In general, however, there is no operation satisfying all of those improvements simultaneously.  This is, of course, no problem if one only needs finitely many improvements.

Yet, there are often good reasons to not even assume that $'$ is a reflexive inverse operation.  To see why this is the case, we first recall another standard definition.  Given a ring $R$, then an element $a\in R$ is \emph{unit-regular} in the ring $R$ when there is some unit inner inverse.  If every element of $R$ has a unit inner inverse, then $R$ is called a \emph{unit-regular ring}.

Put another way, if $R$ is a unit-regular ring, then we may force $a'\in \U(R)$ for each $a\in R$, in which case we say that $'$ is a \emph{unit inner inverse operation}.  It is interesting to observe that these types of inner inverses can also be described equationally.  First, note that if $a\in \U(R)$ and if $a'$ is \emph{any} inner inverse for $a$, then from $aa'a=a$, we get $a'=a^{-1}$. Thus, $'$ is a unit inner inverse operation if and only if
\[
a'a''=1 \ \text{ and }\ a''a'=1.
\]
Surprisingly, these two identities can be expressed by a single identity; namely,
\[
a'a''a''a'=1,
\]
since this forces $a'$ to be both left and right invertible and hence a unit.

When given a unit-regular ring $(R,\, ')$ where $'$ is merely an inner inverse operation, then in general it is impossible to find a derived operation that is a unit inner inverse operation.  Indeed, according to \cite[Proposition 2.15]{OR}, there is no such derived operation that is \emph{uniformly} defined over the rings $\M_n(F)$, for $F$ a field, as the integer $n\geq 1$ varies.  Thus, there is no such derived operation at all on $\prod_{n=1}^{\infty}\M_n(F)$.

We now have two potential routes for improving a given inner inverse.  We may force it to be reflexive, and when working in a unit-regular ring we may force it to be a unit.  Unfortunately, these routes rarely are compatible.

\begin{lemma}
Let $(R,\ ')$ be a regular ring.  Then $'$ is both a reflexive inverse operation and a unit inner inverse operation if and only if $R=0$.
\end{lemma}
\begin{proof}
Any reflexive inverse $0'$ for $0$ satisfies
\[
0'=0'00'=0,
\]
and this can be a unit only in the zero ring.  Conversely, in the zero ring, there is only one inner inverse operation, and it satisfies the stated properties.
\end{proof}

For the following discussion, it is convenient to view ring elements in a slightly different way.  Recall that there is a natural isomorphism $R\cong \End(R_R)$, where an element $a\in R$ corresponds to the endomorphism ``left multiplication by $a$.''  Hereafter, we will freely identity elements of $R$ with their left multiplication action on the free right $R$-module $R_R$.  Thus, by $\im(a)$ we simply mean the right ideal $aR$, and by $\ker(a)$ be mean the right annihilator in $R$.  Note that we are privileging endomorphisms of the \emph{right} $R$-module $R_R$ over endomorphisms of the left $R$-module $_RR$, since from the beginning we made a choice to work in the category $\FP(R)$ of finitely generated projective \emph{right} $R$-modules.

We can classify regular and unit-regular elements according to their behavior as endomorphism, as follows.  (These classifications are part of a more general situation, which appears in the solutions to Exercises 4.14A$_1$ and 4.14C in \cite{LamExercises}.)  An element $a\in R$ is regular if and only if $\im(a)\subseteq^{\oplus}R_R$.  Letting $a'$ be any inner inverse, then $aa'$ is an idempotent, and
\[
\im(a)=aR=aa'R\subseteq^{\oplus}R_R.
\]
Also, $a'a$ is an idempotent, and
\[
\ker(a)={\rm ann}_r(a)=(1-a'a)R\subseteq^{\oplus}R_R.
\]
Continuing to assume that $a$ is regular, then it is unit-regular in $R$ if and only if
\[
\ker(a)\cong \coker(a):=R/aR,
\]
or in other words $(1-a'a)R\cong (1-aa')R$.  The complementary idempotents $aa'$ and $1-aa'$, as well as their opposites $a'a$ and $1-a'a$, thus play an important role.

Following \cite[Proposition 21.20]{Lam}, two idempotents $e,f\in R$ said to be \emph{isomorphic} (also, sometimes called \emph{equivalent}) when $eR\cong fR$, or equivalently $Re\cong Rf$.  Thus, the two idempotents $a'a$ and $aa'$ are always isomorphic (via left multiplication by $a$), but the complementary idempotents $1-a'a$ and $1-aa'$ are isomorphic exactly when $a$ is unit-regular.

An element $b\in R$ is an inner inverse to $a$ if and only if (left multiplication by) $b$ takes the summand $\im(a)$ isomorphically to a direct sum complement of $\ker(a)$ in $R_R$ and left multiplication by $a$ is the inverse isomorphism.  Nonuniqueness of inner inverses boils down to the fact that $b$ can act arbitrarily on a direct sum complement of $\im(a)$.  To be a reflexive inverse, $\ker(b)$ must be a direct sum complement of $\im(a)$.  In other words, $a,b\in R$ are a pair of reflexive inverses if and only if
\begin{itemize}
\item $R_R=\im(b)\oplus \ker(a)=\im(a)\oplus \ker(b)$ and
\item left multiplications by $a$ and $b$ yield inverse isomorphisms between $\im(b)$ and $\im(a)$.
\end{itemize}

\section{Defining separativity in a new way}\label{Section:NewInnerInverses}

The topic at hand is separative cancellation \eqref{Eq:SeparativityDef}.  However, this is a complicated condition, involving many isomorphisms.  Instead of immediately characterizing that situation in terms of regularity, we first begin with the following two simpler forms of cancellation.

\begin{definition}
A ring $R$ is \emph{core strongly separative} when
\[
R_R\cong A\oplus C\oplus C\cong B\oplus C\ \Rightarrow\ A\oplus C\cong B,
\]
for any $A,B,C\in \FP(R)$.  Moreover, $R$ is \emph{strongly separative} when
\[
A\oplus C\oplus C\cong B\oplus C\ \Rightarrow\ A\oplus C\cong B,
\]
for any $A,B,C\in \FP(R)$.
\end{definition}

Any strongly separative ring is core strongly separative, but the converse can fail.

\begin{example}\label{Example:CorSSepNotSSep}
There exists a core strongly separative ring that is not strongly separative.  Moreover, we can force the ring to be regular.
\end{example}
\begin{proof}
Let $M$ be the commutative (additive) monoid generated by two elements $u$ and $x$, subject to the two relations
\[
u+x=u\ \text{ and }\ x+x=x.
\]
Any element of this monoid is thus uniquely represented as one of the elements in
\[
\{x\}\cup\{ku\, :\, k\in \Z_{\geq 0}\}.
\]
(Define a homomorphism $M\to \N$ by sending $x\mapsto 0$ and $u\mapsto 1$; this shows that the multiples of $u$ are distinct, and all are not equal to $x$ except perhaps $0$.  Similarly, if $N$ is the two element commutative monoid $\{0,\infty\}$ subject to the rule $\infty+z=z$ for any $z\in N$, then the map $M\to N$ sending $x\mapsto \infty$ and $u\mapsto \infty$ shows that $x$ is nonzero.  Alternatively, one could use a normal form theorem, like Bergman's diamond lemma.)

Given any $y,z\in M$, we see (by case analysis, or using the monoid $N$ above) that
\[
y+z=0\ \Rightarrow y=z=0.
\]
This means that $M$ is \emph{conical}.  A similarly easy computation by cases shows that $M$ is a \emph{refinement} monoid as defined in \cite{AGOP}.  Thus, by \cite[Theorem B]{ABP}, together with the paragraph that follows that theorem, we know that $M$ is isomorphic to the monoid $V(R)$, for some (unital) regular ring $R$.

The ring $R$ is not strongly separative, since on the monoid side $2x=x$ does not imply $x=0$.  However, $R$ is core strongly separative, which we see as follows.  The isomorphism class $[R_R]$ can be made to correspond to $u$ (from the paragraph following \cite[Theorem B]{ABP}).  Suppose that
\[
u=a+c+c=b+c
\]
for some $a,b,c\in M$.  If $c=0$, then $a+c=b$ and we are done.  If $c=x$, then we must have $a=b=u$, and again $a+c=b$.  Finally, if $c=ku$ for some integer $k\geq 1$, then we have $u=a+2c=a+2ku$, which is impossible.
\end{proof}

Core strong separativity is defined in terms of modules.  However, we can reexpress it ring theoretically.  To do so, we first introduce an important class of elements in rings, in the following lemma--definition.

\begin{lemma}\label{Lemma:RightRepeaterDef}
Let $R$ be a ring, and let $x\in R$ be a regular element with inner inverse $x'\in R$.  The following are equivalent:
\begin{itemize}
\item[\textup{(1)}] $xR\subsetsim^{\oplus} R/xR$, or in other words $\im(x)$ is isomorphic to a direct summand of $\coker(x)$.
\item[\textup{(2)}] If $xR\oplus Y=R_R$ for some submodule $Y\subseteq R_R$, then $xR\subsetsim^{\oplus}Y$.
\item[\textup{(3)}] $xx'R\subsetsim^{\oplus} (1-xx')R$.
\item[\textup{(4)}] $Rxx'\subsetsim^{\oplus} R(1-xx')$.
\end{itemize}
Any regular element satisfying these conditions is called a {\bf right repeater}.
\end{lemma}
\begin{proof}
The regularity hypothesis on $x$ guarantees that $xR\subseteq^{\oplus}R_R$.  Every direct summand complement of $xR$ in $R_R$ is isomorphic to $R/xR$.  Also, $(1-xx')R$ is one of those direct sum complements.  As the relation $\subsetsim^{\oplus}$ is an isomorphism invariant, this proves the equivalence of the first three conditions.  Finally, from the left-right symmetry of isomorphic idempotents, \cite[Proposition 21.20]{Lam}, we see that $(3)\Leftrightarrow (4)$, since both conditions assert that the idempotent $xx'$ is isomorphic to a subidempotent of $(1-xx')$.
\end{proof}

An important class of examples of right repeaters, which we suggest that the reader keep in mind, are the following.  Given an integer $n\geq 1$ and a field $F$, suppose that $R=\M_n(F)$.  Then a matrix $x\in R$ is a right repeater exactly when ${\rm rank}(x)\leq n/2$.

Returning to the situation where $R$ is an arbitrary ring, assume that the premise of core strong separativity, $R_R\cong A\oplus C\oplus C\cong B\oplus C$, is instantiated by some $A,B,C\in \FP(R)$.  After replacing these modules with isomorphic copies, and renaming as necessary, we may assume that there are internal direct sum decompositions $R=A\oplus C\oplus C_1=B\oplus C_2$, where $C\cong C_1\cong C_2$.  Letting $x\in R$ be any endomorphism that annihilates $B$, and that sends $C_2$ isomorphically to $C_1$, we see that $x$ is a right repeater.

Conversely, let $x\in R$ be any right repeater.  By Lemma \ref{Lemma:RightRepeaterDef}(3), write $(1-xx')R=A\oplus C$ with $C\cong C_1:=xR$.  Putting $B:=(1-x'x)R$ and $C_2:=x'xR$, we now have
\[
R_R=A\oplus C\oplus C_1=B\oplus C_2,
\]
with $C\cong C_1\cong C_2$.  Thus right repeaters are exactly what is needed to instantiate the premise of the implication defining core strong separativity.  Moreover, the conclusion of that implication is $(1-xx')R=A\oplus C\cong B=(1-x'x)R$.  So, the conclusion holds if and only if the right repeater is unit-regular.  In other words:

\begin{prop}\label{Prop:BasicStatementAboutRightRepeaters}
A ring is core strongly separative if and only if all right repeaters in the ring are unit-regular.
\end{prop}

We will do better than this, by showing that unit-regularity of a special subset of the right repeaters suffices.  We will work in an arbitrary ring, where some elements of $R$ may fail to have any inner inverse, but we will still find it convenient to fix a unary operation $'$ on $R$ such that $aa'a=a$ whenever $a$ happens to be regular.  For simplicity, we refer to $'$ as an inner inverse operation on the regular elements.

\begin{prop}\label{Prop:SpecialImageRepeaters}
Let $R$ be a ring, and fix an inner inverse operation $'$ on the regular elements of $R$.  If $a\in R$ is regular, and if $b\in R$ is such that $(1-aa')baa'$ is regular, then
\[
x:=(1-aa')ba
\]
is a right repeater.
\end{prop}
\begin{proof}
Assuming that $a$ is regular, then $aR=aa'R$ and hence
\[
xR=(1-aa')baR=(1-aa')baa'R.
\]
Assuming $(1-aa')baa'$ is regular, then $xR$ is a direct summand of $R_R$.  Thus, $x$ is regular.  Since $xR\subseteq^{\oplus} (1-aa')R$, we may fix an inner inverse $x^{\dagger}$ for $x$, possibly different from $x'$, such that $aa'R\subseteq^{\oplus}(1-xx^{\dagger})R$.

Now, $xx^{\dagger}=(1-aa')bax^{\dagger}$ is an idempotent.  Hence, $ax^{\dagger}xx^{\dagger}(1-aa')b$ is an isomorphic idempotent.  (Here we are using the easy fact, called the ``flip-and-double trick,'' that if $cd$ is an idempotent, then $dcdc$ is also an idempotent isomorphic to $cd$.  Take $c:=(1-aa')b$ and $d:=ax^{\dagger}$.)  Therefore, $xR$ is isomorphic to the summand
\[
ax^{\dagger}xx^{\dagger}(1-aa')bR \subseteq^{\oplus}aR= aa'R\subseteq^{\oplus} (1-xx^{\dagger})R.\qedhere
\]
\end{proof}

\begin{remark}\label{Remark:ExplicitIsos}
With access to $x$ and to the parameters $a$ and $b$, the previous proof explicitly constructs an isomorphism, using the flip-and-double trick, from $\im(x)$ to a summand in a complement of $\im(x)$.
\end{remark}

We are now ready to limit the right repeaters needed to establish core strong separativity.

\begin{thm}\label{Thm:ImageRepeatUnitRegCor}
Let $R$ be a ring, and fix a distinguished inner inverse operation $'$ on the set of regular elements of $R$.  Then $R$ is core strongly separative if and only if the right repeaters from \textup{Proposition \ref{Prop:SpecialImageRepeaters}} are all unit-regular.
\end{thm}
\begin{proof}
By Proposition \ref{Prop:BasicStatementAboutRightRepeaters}, it suffices to show the converse.  Assume every right repeater of that special form is unit-regular.  Also, assume that $R_R=A\oplus C\oplus C_1=B\oplus C_2$ for some right ideals $A,B,C,C_1,C_2\subseteq R_R$, with $C\cong C_1\cong C_2$.

Let $a\in R$ be any endomorphism with kernel $B$ that sends $C_2$ isomorphically to $C$.  Notice that $(1-aa')R$ is a direct sum complement of $\im(a)$, as is $A\oplus C_1$, and so $(1-aa')R\cong A\oplus C_1$.  Write $(1-aa')R=A'\oplus C_1'$ with $A\cong A'$ and $C_1\cong C_1'$.

Let $b$ be any endomorphism with kernel $(1-aa')R$, that sends $C$ isomorphically to $C_1'$.  Thus, $b=(1-aa')baa'$.  Also $b$ is regular, having an image that is a direct summand.

Now, putting $x:=(1-aa')ba$, we see that this is an endomorphism with kernel $B$ that sends $C_2$ isomorphically to $C_1'$.  Since $A'\oplus C$ is a direct sum complement of $C_1'$ in $R_R$, then
\[
A\oplus C\cong A'\oplus C\cong \coker(x)\cong \ker(x)=B,
\]
where the third isomorphism follows from the unit-regularity of $x$.
\end{proof}

\begin{cor}\label{Cor:SpecialImageRepeatersUnitReg}
A regular ring $(R,\ ')$ is core strongly separative if and only if $(1-aa')ba$ is unit-regular, for every $a,b\in R$.
\end{cor}

Another way to think about this corollary is that core strongly separative regular rings have special inner inverse operations, where every element of the form $(1-aa')ba$ has a unit inner inverse.

Having dealt with the special case of \emph{core} strong separativity, let us now handle the non-core case, which is the more natural class.  We will find it convenient to restrict this classification to the class of exchange rings (which, importantly, contains the regular rings).  A ring $R$ is an \emph{exchange ring} when the right $R$-module $R_R$ satisfies the finite exchange property in direct sums of modules.  Alternative characterizations for exchange rings appear throughout the literature, but see \cite{Nicholson} for an introduction to this topic.

The following lemma describes a nice characterization of strong separativity (and separativity) in exchange rings, which in retrospect partly explains the ``core'' restriction.

\begin{lemma}\label{Lemma:MainRestriction}
Let $R$ be an exchange ring.  The following hold:
\begin{itemize}
\item[\textup{(1)}] $R$ is strongly separative if and only if
\[
A\oplus C\oplus C\cong B\oplus C\subsetsim R_R \ \Rightarrow\ A\oplus C\cong B,
\]
for all $A,B,C\in \FP(R)$.
\item[\textup{(2)}] $R$ is separative if and only if
\[
A\oplus C\oplus C \cong B\oplus C\oplus C\subsetsim R_R\ \Rightarrow\ A\oplus C\cong B\oplus C,
\]
for all $A,B,C\in \FP(R)$.
\end{itemize}
\end{lemma}
\begin{proof}
Part (2) is exactly Proposition 2.8 and Corollary 2.9 from \cite{AGOP}.  (There is a minor typo in the proof, where the reference to part (i) of Lemma 2.7 should be to part (ii).)  Part (1) follows from the same argument, \emph{mutatis mutandis}.
\end{proof}

\begin{thm}\label{Thm:CornersStrongSeparativeCore}
An exchange ring $R$ is strongly separative if and only if $eRe$ is core strongly separative for each idempotent $e\in R$.
\end{thm}
\begin{proof}
$(\Rightarrow)$: As pointed out on page 125 of \cite{AGOP}, strong separativity in exchange rings passes to corner rings.  Thus, these corners are also core strongly separative.

$(\Leftarrow)$: Assume all corners are core strongly separative.  Also assume that
\[
A\oplus C\oplus C\cong B\oplus C\subsetsim R_R.
\]
After replacing these modules with isomorphic copies, we may assume there are internal direct sum decompositions
\begin{equation}\label{Eq:CornerReduct}
A\oplus C\oplus C_1=B\oplus C_2=X\subseteq^{\oplus} R_R
\end{equation}
where $C\cong C_1\cong C_2$.  Write $X=eR$ for some idempotent $e\in R$.

We need a (very) little Morita theory; we include the details for completeness.  If $fR\subseteq^{\oplus} eR$ for some idempotent $f\in R$, then we can replace $f$ by $efe$, which is another idempotent with $efeR=fR$.  Thus, given summands $f_1R,f_2R\subseteq^{\oplus}eR$, for some idempotents $f_1,f_2\in R$, then
\[
f_1R\cong f_2R\ \text{ if and only if }\ ef_1eR\cong ef_2eR.
\]
The latter isomorphism is evidenced by left multiplications via two (reflexive inverse) elements of the corner ring $eRe$, by \cite[Proposition 21.20(2)]{Lam}.

There is a natural correspondence between the right ideals of $R$ contained in $eR$ and the right ideals of $eRe$; see \cite[Theorem 21.11(1)]{Lam}.  The argument from the previous paragraph shows that this correspondence respects isomorphisms on cyclic projective modules.  Applying this correspondence to \eqref{Eq:CornerReduct}, we have
\[
Ae\oplus Ce\oplus C_1e=Be\oplus C_2e=eRe,
\]
with $Ce\cong C_1e\cong C_2e$.  By the core strong separativity of $eRe$, we have $Ae\oplus Ce\cong Be$.  Applying the correspondence again, $A\oplus C\cong B$.
\end{proof}

In order to translate the previous theorem into the language of right repeaters, we first need to understand how inner inverses behave in corner rings.  Given any ring $R$, suppose that $a\in R$ is regular, with an inner inverse $a'$.  Further, suppose that $e\in R$ is an idempotent, and that $a\in eRe$.  In particular, $eae=a$ and hence $a'=(eae)'$.  Now, $a'$ might not belong to the corner ring $eRe$, but we do have
\[
a(ea'e)a=aa'a=a.
\]
Thus, $a$ is not only regular in $R$, but also in the corner ring $eRe$, with an inner inverse $ea'e\in eRe$.  Of course, conversely, if $a\in eRe$ is regular in $eRe$, then it is also regular in the overring $R$.

In other words, given any inner inverse operation $'$ on the regular elements of a ring $R$, there is a derived inner inverse operation on all of the regular elements of the corner ring $eRe$, defined by $a^{\dagger}:=ea'e$.  With this understanding in place, we can now classify strongly separative rings in terms of regularity and unit-regularity.

\begin{thm}\label{Thm:FullStronglySepOps}
Let $R$ be a ring with a distinguished inner inverse operation $'$ on its regular elements.  Then $R$ is strongly separative if and only if for each idempotent $e\in R$, for each regular element $a\in eRe$, and for each $b\in eRe$ with $(1-aa')baa'$ regular, then $(1-aa')ba$ is unit-regular in $eRe$.
\end{thm}
\begin{proof}
$(\Leftarrow)$:  Let $e\in R$ be an idempotent, and let $a,b\in eRe$ with $a$ regular and with $(1-aa')baa'$ regular.  Since $(1-aa')baa'eaa'=(1-aa')baa'aa'=(1-aa')baa'$, we have
\[
(1-aa')ba(ea'e)R=(1-aa')baa'R\subseteq^{\oplus}R_R
\]
Moreover, using the notation $a^{\dagger}=ea'e$, then since $a,b\in eRe$ we have
\[
(1-aa')ba(ea'e)=(e-a(ea'e))ba(ea'e)=(e-aa^{\dagger})baa^{\dagger}\in eRe,
\]
and hence it is regular in that corner ring.  Clearly, $(1-aa')ba=(e-aa^{\dagger})ba$.

Thus, as $a$ and $b$ range over the elements of $eRe$, subject to the fact that $a$ is regular and $(1-aa')baa'$ is regular, then the quantity $(1-aa')ba$ ranges over the special right repeaters of Proposition \ref{Prop:SpecialImageRepeaters}, but in $(eRe,\, ^{\dagger})$.  As all of these elements are unit-regular by assumption, then by Theorem \ref{Thm:ImageRepeatUnitRegCor} every corner ring of $R$ is core strongly separative. Theorem \ref{Thm:CornersStrongSeparativeCore} then implies that $R$ is strongly separative.

$(\Rightarrow)$:  Assume $R$ is strongly separative.  Then by Theorem \ref{Thm:CornersStrongSeparativeCore}, every corner ring is core strongly separative, and hence all right repeaters in those corner rings are unit-regular (in that corner).  The computation from the first paragraph of this proof, together with Proposition \ref{Prop:SpecialImageRepeaters}, shows that the elements we are considering are indeed right repeaters in the respective corner rings.
\end{proof}

\begin{cor}\label{Cor:RegStrongSepUReg}
Let $R$ be a regular ring.  Then $R$ is strongly separative if and only if for each idempotent $e\in R$, and for all $a,b\in eRe$, then $(1-aa')ba$ is unit-regular in $eRe$.
\end{cor}

Left repeaters are defined in the obvious way, by using a left-right symmetric version of Lemma \ref{Lemma:RightRepeaterDef}.  The previous work in this section can now be repeated for left repeaters, with very little changes to the proofs.  In regular rings, elements of the form $ac(1-a'a)$ are always left repeaters, and a regular ring is core strongly separative if and only if these special left repeaters are all unit-regular.  Using the bullet points at the very end of Section \ref{Section:InnerInverses}, one can show that left repeaters and right repeaters come in reflexive inverse pairs.

In matrix rings over fields, the left repeaters and right repeaters are exactly the same elements; namely, those whose rank is no more than half the number of rows.  In the ring of column-finite matrices over a field, these classes of elements are different, as
\[
E_{2,1}+E_{4,2}+E_{6,3}+\cdots = \begin{bmatrix}0 & \phantom{\ddots} & \phantom{\ddots} & \phantom{\ddots} & \phantom{\ddots}\\
1 & 0 & \phantom{\ddots} & \phantom{\ddots} & \phantom{\ddots}\\
\phantom{\ddots}& 0 & 0 & \phantom{\ddots} & \phantom{\ddots}\\
 \phantom{\ddots}& 1 & 0 & 0 & \phantom{\ddots}\\
 \phantom{\ddots}& \phantom{\ddots}& \phantom{\ddots}& \phantom{\ddots}&  \ddots \end{bmatrix}
\]
is an right repeater but not a left repeater.

Having dealt with (core) strong separativity, we now can tackle separativity.

\begin{definition}\label{Def:CoreSep}
A ring $R$ is \emph{core separative} when
\[
R_R\cong A\oplus C\oplus C\cong B\oplus C\oplus C\ \Rightarrow\ A\oplus C\cong B\oplus C,
\]
for any $A,B,C\in \FP(R)$.  In case $R$ is an exchange ring, then it is \emph{separative} when
\[
A\oplus C\oplus C\cong B\oplus C\oplus C\ \Rightarrow\ A\oplus C\cong B\oplus C,
\]
for all $A,B,C\in \FP(R)$.
\end{definition}

This definition of separativity is equivalent to the one in the introduction---for \emph{exchange rings}---by using \cite[Lemma 2.1 and p.\ 115]{AGOP}.

As we have no examples of exchange rings that are not separative, we do not know whether the two classes of rings in Definition \ref{Def:CoreSep} are distinct for exchange rings.  However, we do have:

\begin{example}
There is a core separative ring $R$, and some modules $A,B,C\in \FP(R)$ satisfying $A\oplus C\oplus C\cong B\oplus C\oplus C$, and yet $A\oplus C\not\cong B\oplus C$.
\end{example}
\begin{proof}
From page 114 of \cite{AGOP}, there is a commutative (additive) monoid $M$ that is conical, has the refinement property, but is not separative (where separativity is a direct translation of Definition \ref{Def:CoreSep} to refinement monoids via \cite[Lemma 2.1]{AGOP}).  Let $M'=M\sqcup \{\infty\}$, which is still conical, has the refinement property, and fails separativity.

From the paragraph after Theorem 3.4 in \cite{BergmanDicks}, we know that $M'\cong V(R)$, for some ring $R$, where $\infty$ corresponds to $[R_R]$.  To show that $R$ is a core separative ring, it suffices to show that the corresponding condition in $M'$ holds, with $[R_R]$ replaced by $\infty$.  Thus, suppose that $\infty=a+c+c=b+c+c$ for some $a,b,c\in M'$.  Then either $c=\infty$ or $a=b=\infty$.  In any case, $a+c=b+c$.
\end{proof}

Just as right repeaters instantiate the premise of (the implication defining) core strong separativity, the following elements instantiate the premise of core separativity.

\begin{definition}
Let $R$ be a ring, and let $x\in R$.  We say that $x$ is a (two-sided) \emph{repeater} if it is both a left and right repeater.
\end{definition}

Describing a special collection of repeaters in general rings is slightly more complicated than it was for right repeaters.  This is done as follows:

\begin{prop}\label{Prop:SpecialRepeaters}
Let $R$ be a ring, and fix an inner inverse operation $'$ on its regular elements.  Asume that $a\in R$ is regular, and that $c\in R$ is such that $a'ac(1-a'a)$ is regular.  Setting $y:=ac(1-a'a)$, further suppose that $b\in R$ is such that $(1-yy')byy'$ is regular.  Then
\[
x:=(1-yy')bac(1-a'a)
\]
is a repeater.
\end{prop}
\begin{proof}
By the left-right symmetric version of Proposition \ref{Prop:SpecialImageRepeaters}, we know that $y$ is a left repeater, and hence regular.  Thus, by a second application of Proposition \ref{Prop:SpecialImageRepeaters}, we know that $x=(1-yy')by$ is an right repeater.  It just remains to show that $x$ is a left repeater.

Notice that $a'aR\subseteq \ker(x)$.  By the ``flip-and-double trick'' (mentioned in the proof of Proposition \ref{Prop:SpecialImageRepeaters}, but this time with $c:=x'(1-yy')ba$ and $d:=a'ac(1-a'a)$, so that $cd=x'x$), we then know that
\[
x'xR\cong a'ac(1-a'a)x'xx'(1-yy')baR\subseteq^{\oplus}\ker(x)=(1-x'x)R. \qedhere
\]
\end{proof}

We can modify Proposition \ref{Prop:BasicStatementAboutRightRepeaters} and Theorem \ref{Thm:ImageRepeatUnitRegCor} to work for core separativity, by essentially the same proofs.

\begin{thm}\label{Thm:SpecialRepeaters2}
Let $R$ be a ring, and fix a distinguished inner inverse operation $'$ on the set of regular elements of $R$.  Then $R$ is core separative if and only if all repeaters in $R$ are unit-regular, if and only if the repeaters from \textup{Proposition \ref{Prop:SpecialRepeaters}} are all unit-regular.
\end{thm}

The following lemma will help us give a simpler set of repeaters in regular rings.

\begin{lemma}\label{Lemma:SpecialRegularRepeaters}
If $(R,\, ')$ is a regular ring, then any right multiple of a right repeater is a right repeater. Similarly, left multiples of left repeaters are left repeaters.

Consequently, for any $a,b,c\in R$, then $(1-aa')bac(1-a'a)$ is a repeater.
\end{lemma}
\begin{proof}
Assume that $x\in R$ is a right repeater, so $xR$ is isomorphic to a summand in $(1-xx')R$.  Let $y\in R$ be arbitrary.  Now, since $xy$ is regular, we know that $xyR\subseteq^{\oplus}R_R$.  Thus, $xyR\subseteq^{\oplus}xR$.  Hence, $xyR$ is isomorphic to a summand in $(1-xx')R$.  As $(1-xx')R$ is a direct summand in a complement of $xyR$, we have $xyR$ is isomorphic to a summand in a complement of $xyR$.

The second sentence of the lemma follows by symmetry.

Since $(1-aa')bac(1-a'a)$ is a right multiple of the right repeater $(1-aa')ba$, it is also a right repeater.  By symmetry, it is also a left repeater.
\end{proof}

\begin{cor}\label{Cor:CoreSepReg}
A regular ring $R$ is core separative if and only if $(1-aa')bac(1-a'a)$ is unit-regular, for every $a,b,c\in R$.
\end{cor}
\begin{proof}
The repeaters in the statement of Theorem \ref{Thm:SpecialRepeaters2} can be replaced with the simpler ones from Lemma \ref{Lemma:SpecialRegularRepeaters}, using essentially the same proof.
\end{proof}

Our work with corner rings also generalizes from the strongly separative case to the separative case.  In particular, (again, by modifying proofs in trivial ways, as needed) we have:

\begin{thm}
An exchange ring $R$ is separative if and only if $eRe$ is core separative for each idempotent $e\in R$.
\end{thm}

We thus arrive at a new characterization of separative regular rings.

\begin{cor}\label{Cor:RegSepUReg}
A regular ring $(R,\, ')$ is separative if and only if for every idempotent $e\in R$, and for every $a,b,c\in eRe$, the element
\[
(1-aa')bac(1-a'a)
\]
is unit-regular in $eRe$.
\end{cor}

By passing to a corner ring, if necessary, we thus see that the separativity problem is equivalent to deciding whether or not every element in a regular ring of the form
\[
(1-aa')bac(1-a'a)
\]
is unit-regular.

\section{Viewing algebras as varieties}\label{Sec:FirstVariety}

To begin this section we quickly review the notion of a \emph{variety}, in the sense of universal algebra. A variety is a collection of algebras defined using two pieces of information.  First, we need a set of operations, called the \emph{signature}.  For instance, the signature of rings is usually given by $\{+,\cdot, -,0,1\}$, where $+,\cdot$ are binary operations, $-$ is a unary operation, and $0,1$ are nullary operations (i.e., constants).  Second, we need a set of universal identities satisfied under these operations.  The standard identities for rings, each universally quantified over the variables, are:
\begin{eqnarray*}
x+y=y+x & & \text{(commutative law for addition)}\\
(x+y)+z=x+(y+z) & & \text{(associative law for addition)}\\
x+0=x && \text{($0$ is a right additive identity)}\\
x+(-x)=0& & \text{(negation yields a right additive inverse)}\\
(x\cdot y)\cdot z=x\cdot(y\cdot z)& & \text{(associative law for multiplication)}\\
x\cdot 1=x & & \text{($1$ is a right multiplicative identity)}\\
1\cdot x=x & & \text{($1$ is a left multiplicative identity)}\\
x\cdot(y+z)=(x\cdot y)+(x\cdot z)& & \text{(left distributivity of multiplication over addition)}\\
(y+z)\cdot x=(y\cdot x)+(z\cdot x)& & \text{(right distributivity of multiplication over addition)}
\end{eqnarray*}
The list of identities will sometimes vary, since different combinations of identities will give the same consequences.  For instance, we didn't include the fact $0$ is a left additive identity because this follows from commutativity of addition in conjunction with $0$ being a right additive identity.  Also, the signature of rings will sometimes vary, as we can (for instance) avoid referring to $0$ if we like.  It is often more important to pick a convenient representation, rather than attempting to minimize the axioms and signature.

The variety of rings is the collection of $6$-tuples $(R,+,\cdot,-,0,1)$, where $R$ is a set, and the other symbols are operations on $R$ satisfying the identities above.  We will denote this variety as {\bf Ring}.

Varieties form a category, where a morphism is a map that respects each of the operations.  For instance, a morphism in the variety of rings $(R,+,\cdot,-,0,1)\to (S,+,\cdot,-,0,1)$ is a map $f\colon R\to S$ where
\[
f(x+y)=f(x)+f(y),\ f(x\cdot y)=f(x)\cdot f(y),\ f(-x)=-f(x),\ f(0)=0,\ \text{and}\ f(1)=1.
\]
We refer to $f$ as a {\bf Ring}-homomorphism when emphasizing its connection to the variety.

Beginning students in algebra sometimes learn a definition of rings where, say, the axiom of ``right multiplicative identity'' is given as $\exists y\, \forall x\, (x\cdot y=x)$.  However, when defining varieties we are not allowed to use existential quantification, nor implications, nor other Boolean connectives.  Thus, varieties are sometimes referred to as \emph{equational classes}.  Fortunately, the axioms of rings are amenable to being rewritten as pure identities, after adjoining a few additional operations, such as the nullary operation $1$.

The universal formula
\[
x\cdot y=0\Rightarrow (x=0 \lor y=0),
\]
which asserts that there are no nonzero zero-divisors, is not a pure identity.  Moreover, unlike for rings, the class of domains cannot be made into a variety, due to:

\begin{thm}[{Birkhoff's theorem; see \cite[Theorem 9.6.1]{BergmanBook}}]
Given any signature, a collection of algebras under that signature forms a variety if and only if the collection is closed under homomorphic images, subalgebras, and direct products.
\end{thm}

Powerful results like Birkhoff's theorem make the study of varieties especially interesting.  Another such result, which we will make use of in this paper, is the following:

\begin{thm}[{\cite[Theorem 9.3.7]{BergmanBook}}]
Given any variety and given any set $X$, there is a free algebra on the set $X$ in the variety.
\end{thm}

With this short introduction out of the way, we now connect these ideas to the work in Section \ref{Section:NewInnerInverses}.  We can view the class of regular rings with a distinguished inner inverse operation as a variety, denoted {\bf Reg}, by using the same signature as for rings, except that we adjoin an extra unary operation $'$.  The identities are the same as for rings, except that we adjoin one extra identity
\[
aa'a=a,
\]
which forces $a'$ to be an inner inverse for $a$.  Unless context suggests otherwise, other varieties in this section are understood to include the signatures and identities of {\bf Reg}.  Whenever possible, we will use exactly the same signature, and there is a special name for this situation.

\begin{definition}
Given a variety {\bf V}, a \emph{subvariety} is another variety with the same signature as {\bf V} that is subject to the identities of {\bf V}, but there may also be additional identities.
\end{definition}

For example, by adjoining the one new identity
\[
a'aa'=a'
\]
we have a subvariety (of {\bf Reg}) where the distinguished inner inverse operation always yields a reflexive inverse.  Similarly, unit regular rings (or, more properly, regular rings with a distinguished unit inner inverse) form a subvariety by adjoining the one new identity
\[
a'a''a''a'=1.
\]

Thinking of the symbols $a$ and $b$ as parameters, we can put $x:=(1-aa')ba$.  Then the core strongly separative regular rings form a subvariety by adjoining the one new identity
\[
x'x''x''x'=1,
\]
which forces the special right repeaters of Corollary \ref{Cor:SpecialImageRepeatersUnitReg} to be unit-regular.  Similarly, putting $x:=(1-aa')bac(1-a'a)$, and asserting $x'x''x''x'=1$, we have defined the core separative regular rings as a subvariety, according to Corollary \ref{Cor:CoreSepReg}.

Our next goal is to explain how to similarly varietize the strongly separative, as well as separative, regular rings.  We need a uniform method for referring to idempotents.  This is done, as follows:

\begin{lemma}\label{Lemma:RegularIdempotents}
Given a regular ring $(R,\, ')$, the set of idempotents of $R$ equals the set of elements of the form $e:=aa'+a(1-aa')$ for some $a\in R$.
\end{lemma}
\begin{proof}
First, since $aa'a=a$, we have $(1-aa')a=0$.  Thus
\[
e^2=aa'aa'+aa'a(1-aa') + a(1-aa')aa' + a(1-aa')a(1-aa')=aa'+a(1-aa')=e.
\]
Conversely, let $f\in R$ be an arbitrary idempotent.  Taking $a=f$, then
\[
e=ff'+f(1-ff')=ff'+f-ff'=f.
\]
So the expression represents all idempotents.
\end{proof}

We also must describe unit-regularity in corner rings.  The following result of Lam and Murray from \cite{LamMurray} is exactly what is needed to transfer that information back to the original ring.

\begin{lemma}\label{Lemma:ShiftBackCorner}
Let $R$ be a ring, let $e\in R$ be an idempotent, and let $x\in eRe$.  Then $x$ is unit-regular in $eRe$ if and only if $x+(1-e)$ is unit-regular in $R$.
\end{lemma}

Now, using three parameters $a,b,e$, where $e$ is an idempotent, set
\[
x_e:=(1-eae(eae)')ebeae +(1-e).
\]
A ring is strongly separative exactly when all $x_e$ are unit-regular in $R$, by Corollary \ref{Cor:RegStrongSepUReg}.  So, we may varietize strongly separative regular rings by merely asserting
\[
x_e'x_e''x_e''x_e'=1.
\]
Replacing $x_e$ with the four parameter expression
\begin{equation}\label{Eq:xe}
x_e:=(1-eae(eae)')ebeaece(1-(eae)'eae) +(1-e),
\end{equation}
and asserting the same relation, we have varietized the separative regular rings, according to Corollary \ref{Cor:RegSepUReg}.

To end this section, we mention three useful applications of these varieties.  First, they point us towards a concrete (but somewhat complicated) regular ring that could help to answer the separativity problem.

\begin{thm}\label{Thm:ThreeGensNotSepIfAny}
If there is a nonseparative regular ring, then the free {\bf Reg}-algebra on three \textup{(}or more\textup{)} generators is such a ring.
\end{thm}
\begin{proof}
Assume there is a nonseparative regular ring $(R,\, ')$.  After passing to a corner if necessary, then by Corollary \ref{Cor:RegSepUReg} there are three elements $a,b,c\in R$ such that the repeater $x:=(1-aa')bac(1-a'a)$ is not unit-regular in $R$.

Let $S$ be the free {\bf Reg}-algebra on three generators $p,q,r$.  There is a unique {\bf Reg}-homomorphism $\varphi\colon S\to R$ determined by sending $p,q,r$ respectively to $a,b,c$.  Fixing $y:=(1-pp')qpr(1-p'p)\in S$, then $\varphi(y)=x$.  Now $y$ is not unit-regular in $S$, since $x$ is not unit-regular in $R$.  But $y$ is a repeater in $S$, so $S$ is not separative.
\end{proof}

Second, they bring uniformity and explicitness to constructions.  For example, it was known by combining \cite[Theorem 3.4]{AGOPEarly} and \cite[Theorem 2.10]{AGOR} that:

\begin{prop}\label{Prop:CharacterizeExchangeSeparativeRings}
An exchange ring $R$ is separative if and only if for each idempotent $e\in R$, then each regular matrix in $\M_{2}(eRe)$ can be reduced to a diagonal matrix by a finite number of elementary row and column operations over the ring $eRe$.
\end{prop}

We show, in Section \ref{Section:6NowWorks}, that for a (core) separative regular ring $R$, then each matrix $A\in \M_2(R)$ needs only three row operations and three column operations to diagonally reduce.  Moreover, these operations are explicitly definable via regular ring expressions, treating the entries of $A$ as parameters, as long as the inner inverse operation satisfies the new identity used to varietize the core separative regular rings.  More precisely, we only need the unit-regularity of exactly one repeater, easily defined in terms of the entries of $A$.  Similarly, the $2\times 1$ matrices over a (core) strongly separative ring can be explicitly diagonally reduced using the unit-regularity of only one right repeater.

Third, these varieties can be used to strengthen results in the literature, while giving added evidence that there are nonseparative regular rings.  Here is one example.  Fix a field $F$, and let $n\in \Z_{\geq 1}$ be a parameter that we allow to vary.  Recall that in Section \ref{Section:InnerInverses} we mentioned the fact, appearing in \cite[Proposition 2.15]{OR}, that there is no uniformly defined derived operation that is a unit inner inverse operation for all of the rings $\M_n(F)$, in terms of a fixed (but arbitrary) inner inverse operation $'$ on $\M_n(F)$.  The proof shows more; there is no way to diagonally reduce all $2\times 1$ matrices uniformly over these rings.  But, as we mentioned at the end of the previous paragraph, there is a uniform way to diagonally reduce $2\times 1$ matrices in the variety of core strongly separative regular rings.  Thus, there is no uniform way to view $(\M_n(F),\, ')$ as an element of that variety under some derived operation.

In other words, given access to a distinguished (but arbitrary) inner inverse $'$ on $\M_n(F)$, and given a matrix $x\in \M_n(F)$ that we know satisfies ${\rm rank}(x)\leq n/2$ because we also have access to an explicit isomorphism from $xR$ to a summand of $(1-xx')R$, then we cannot uniformly define an isomorphism $(1-x'x)R\to (1-xx')R$ from this information.   The separativity problem reduces to the question of whether we could define such an isomorphism uniformly when given just one additional small piece of information; namely, that $(1-x'x)R$ contains an isomorphic copy of $xR$.  Frankly, this seems unlikely.

There are other properties of separative regular rings, besides the cancellation condition
\[
A\oplus 2C\cong B\oplus 2C\ \Rightarrow\ A\oplus C\cong B\oplus C,
\]
that allow for varietization.  Indeed, as noted on pages 398--399 of \cite{OR}, one can varietize the class of rings where $2\times 2$ matrices are reducible to diagonal matrices, by introducing sixteen new diagonal reduction operations.  Then, after adjoining an inner inverse operation and after forcing the diagonal reduction identities to hold in corner rings, one has an alternative varietal definition of separative regular rings.  Our work in Section \ref{Section:6NowWorks} describes those sixteen operations using derived expressions.

\section{Varieties of exchange rings}\label{Sec:SecondVariety}

Much of the work from the previous section likewise allows us to define a variety of separative exchange rings.  We first make exchange rings a variety, {\bf Exch}.  One way this can be done is by using the description of exchange rings found in \cite{KLN}, as follows.  Adjoin to the signature of rings a unary operation $^{s}$, and adjoin to the identities of rings the equality
\begin{equation}\label{Eq:SuitableEq}
(1-a^s a)(1+(1-a)a^s)=0.
\end{equation}
The element $a^s$ is called a \emph{suitabilizer} for $a$, and the theory of such elements was first studied in \cite{KLN}, although suitabilizers had appeared implicitly in previous papers.  Think of the suitabilizer map as a replacement for the inner inverse operation in regular rings.

Next, we need to show that idempotents and regular elements can be uniformly described in exchange rings.

\begin{lemma}\label{Lemma:ExchangeRegular}
Let $R$ be an exchange ring with suitabilizer map $^s$.
\begin{itemize}
\item[\textup{(1)}] The set of idempotents of $R$ equals the set of elements of the form
\[
e:=(1+a^s(1-a))a
\]
for some $a\in R$.
\item[\textup{(2)}] The set of regular elements of $R$ equals the set of elements of the form
\[
p=p(a,b):=a(1+(ba)^s(1-ba))ba
\]
for some $a,b\in R$.
\end{itemize}
\end{lemma}
\begin{proof}
(1) The complement of the idempotent from \cite[Proposition 3.3(C)]{KLN} is exactly $e$.  Moreover, given an idempotent $f\in R$, then taking $a=f$ we immediately get $e=f$.

(2) First, if we set $g:=(1+(ba)^s(1-ba))ba$, then this is an idempotent by part (1).  Now, taking
\[
q=q(a,b):=(1+(ba)^s(1-ba))b
\]
we find
\[
pqp=ag^3=ag=p.
\]
Thus $p$ is a regular element with inner inverse $q$.  Conversely, suppose $c$ is a regular element of $R$ with an inner inverse $d$.  Taking $a=c$ and $b=d$, then $p=c$, so it represents all regular elements.
\end{proof}

Let $(R,\, ^{s})$ be an exchange ring.  For any two elements $a,b\in R$ define $p(a,b)$ and $q(a,b)$ as in the previous lemma.  Given any regular element $x\in R$, then we can fix (once and for all) elements $a_x,b_x\in R$ such that $p(a_x,b_x)=x$.  Define an inner inverse operation $'$ on the regular elements of $R$ by the rule $x'=q(a_x,b_x)$.  Now, as we run over all pairs of elements $a,b\in R$, not only does $p(a,b)$ run over all regular elements, but $q(a,b)$ runs over enough inner inverses to cover all of the inner inverses under the operation $'$.

Next, suppose that $c\in (1-xx')Rxx'$ is regular.  Then
\[
p((1-xx')cxx',c')=p(c,c')=c.
\]
On the other hand, for any $a,b\in R$, we have $p((1-xx')axx',b)\in (1-xx')Rxx'$.  So, elements of the form $p((1-xx')axx',b)$ represent all regular elements of $(1-xx')Rxx'$.

Putting this all together, given parameters $a,b,c,d$, then
\begin{equation}
z:=p\Big(\big(1-p(a,b)q(a,b)\big)cp(a,b)q(a,b),d\Big)p(a,b)
\end{equation}
is a right repeater, by Proposition \ref{Prop:SpecialImageRepeaters}.  Moreover, as these four parameters range across the elements of a ring, the expression ranges across all the special (and possibly more) right repeaters needed in Theorem \ref{Thm:ImageRepeatUnitRegCor} under some fixed inner inverse operation on that ring.  We can assert that $z$ is unit-regular, thus varietizing the core strongly separative exchange rings, by adjoining two unary operations $u,v$, and by adding the two identities $z u(z) z = z$ and $u(a)v(a)u(a)=1$.

Similar arguments work for strongly separative exchange, core separative exchange, and separative exchange rings; the details are left to the motivated reader.

\begin{question}
Can the regular elements of an algebra in {\bf Exch} be represented by an expression using only one parameter?
\end{question}

\begin{question}
Can the varieties in this section be defined without new operations?
\end{question}

\section{Minimizing the number of elementary operations}\label{Section:MinimizeNumber}

The observation that separative exchange rings are GE-rings, by \cite[Theorem 2.8]{AGOR}, lets us diagonally reduce using elementary matrices, rather than using arbitrary units (which may be quite complicated).  The proof that separative exchange rings are GE-rings requires a large number of elementary transformations.  We will focus on the case of $2\times 2$ matrices, where this number can be significantly reduced, and in fact minimized.  An important fact that often lets us reduce this number, and which will be used repeatedly, is the normalizing equality
\begin{equation}\label{Eq:NormalizingEquality}
\begin{bmatrix}d_1 & 0\\ 0 & d_2\end{bmatrix}\begin{bmatrix}1 & x\\ 0 & 1\end{bmatrix}=\begin{bmatrix}1 & d_1xd_2^{-1}\\ 0 & 1\end{bmatrix}\begin{bmatrix}d_1 & 0\\ 0 & d_2\end{bmatrix}
\end{equation}
for invertible diagonal matrices.  Of course, there is a similar equality involving lower-triangular $2\times 2$ elementary matrices.

The rings where only a small number of elementary operations are needed to diagonally reduce $2\times 2$ matrices can be completely characterized.  Recall that a ring $R$ has \emph{stable rank $1$} when for any $a,b\in R$, if $aR+bR=R$, then there is some $x\in R$ with $a+bx\in \U(R)$.  The following theorem, tying stable rank $1$ to diagonal reduction, appears to be new.

\begin{thm}\label{Thm:DiagBy3Classification}
A ring $R$ has stable rank $1$ if and only if every matrix in ${\rm GL}_2(R)$ can be reduced to a diagonal matrix by at most three elementary operations.
\end{thm}
\begin{proof}
Throughout this proof, let $X\in {\rm GL}_2(R)$ be an arbitrary invertible matrix.

$(\Rightarrow)$: Assume $R$ has stable rank $1$.  Write the first row of $X$ as $[a,b]$ for some $a,b\in R$.  As $XX^{-1}$ is the identity matrix, we have $aR+bR=R$.  Thus, from the definition of stable rank $1$, we may fix some $x\in R$ with $a+bx\in \U(R)$.  This means that after one elementary column operation, we can force the northwest corner of $X$ to be a unit.  Using that unit, we can finish the diagonal reduction after two more elementary operations.

$(\Leftarrow)$: When $X$ can be diagonally reduced by elementary operations, we claim that the reduction can be accomplished by using only column operations, but the same number.  To see why, suppose that $M_1$ and $M_2$ are finite products of elementary $2\times 2$ matrices, and that $D:=M_1 X M_2$ is diagonal.  Then, $X=M_{1}^{-1}DM_{2}^{-1}$.  By using the normalizing identity \eqref{Eq:NormalizingEquality} finitely many times, we have $M_{1}D=DM_3$ for some matrix $M_3$ that is also a product of elementary matrices, with the same number of elementary matrices as in $M_1$.  Rearranging, we obtain $M_{1}^{-1}D=DM_3^{-1}$. Thus, $X=DM_3^{-1}M_2^{-1}$, and so $D=XM_2M_3$.

Now, assume every matrix in ${\rm GL}_2(R)$ can be diagonally reduced by at most three elementary operations.  The identity operation is an elementary operation, so we can use exactly three operations.  Further, note that a product of two upper-triangular elementary matrices is again elementary, since
\[
\begin{bmatrix}1 & p\\ 0 & 1\end{bmatrix}\begin{bmatrix}1 & q\\ 0 & 1\end{bmatrix}=\begin{bmatrix}1 & p+q\\ 0 & 1\end{bmatrix}.
\]
The same is true for lower-triangular elementary matrices.  So, without loss of generality, we may assume that the diagonal reduction operations are represented by alternatingly upper-triangular and lower-triangular matrices (by combining like types as necessary).

We first show that if $[a,b]$ is the top row of an invertible matrix $X$, and if $ay+bz=1$ for some $y,z\in R$, then $[a,b]$ is the top row of a new invertible matrix $Y\in {\rm GL}_2(R)$, where the left column of $Y^{-1}$ is $\left[\begin{smallmatrix}y \\ z\end{smallmatrix}\right]$.  Let $M_1$ be a product of three elementary matrices such that $XM_1=\left[\begin{smallmatrix}v & 0\\ 0 & w\end{smallmatrix} \right]$ for some $v,w\in R$.  Any diagonal matrix in ${\rm GL}_2(R)$ must have unit entries along the diagonal, and so $v,w\in \U(R)$.  Now,
\[
1=\begin{bmatrix}a & b\end{bmatrix}\begin{bmatrix}y \\ z\end{bmatrix}=\begin{bmatrix}a & b\end{bmatrix}M_1 M_1^{-1}\begin{bmatrix}y \\ z\end{bmatrix}
=\begin{bmatrix}v & 0\end{bmatrix}M_1^{-1}\begin{bmatrix}y\\ z\end{bmatrix}.
\]
Thus,
\[
M_1^{-1}\begin{bmatrix}y\\ z\end{bmatrix}=\begin{bmatrix}v^{-1}\\ t\end{bmatrix}
\]
for some $t\in R$.  Fix $M_2=\left[\begin{smallmatrix}1 & 0\\ -tv & 1\end{smallmatrix}\right]$, so $M_2M_1^{-1}\left[\begin{smallmatrix}y\\ z\end{smallmatrix}\right]=\left[\begin{smallmatrix}v^{-1}\\ 0\end{smallmatrix}\right]$.  Note that $[v,0]M_2=[v,0]$.  Fixing $Y:=\left[\begin{smallmatrix}v & 0\\ 0 & 1\end{smallmatrix}\right]M_2M_{1}^{-1}$, then the top row of $Y$ and the left column of $Y^{-1}$ are as desired.

Finally, we show that given $a,b\in R$ with $aR+bR=R$, then there is some element $x\in R$ with $a+bx\in \U(R)$.  Fix any two elements $r,s\in R$ such that $ar+bs=1$.  Since $ar + (bs)1 = 1$, the previous paragraph implies that there is an invertible matrix whose top row is $[a,\, bs]$.

Since $bs=(1-ar)$, we have $(bs)^2=1-2ar+arar$.  Thus,
\[
a(2r-rar)+(bs)(bs)=2ar-arar+(1-2ar+arar)=1.
\]
By another application of the work above, there is some matrix $Y\in {\rm GL}_2(R)$ whose first row is $[a,bs]$, and where the left column of $Y^{-1}$ is $\left[\begin{smallmatrix}2r-rar\\ bs\end{smallmatrix}\right]$.  There are now two cases to consider.

First, suppose that
\[
Y\begin{bmatrix}1 & 0\\ p_1 & 1\end{bmatrix}\begin{bmatrix}1 & p_2\\ 0 & 1\end{bmatrix}\begin{bmatrix}1 & 0\\ p_3 & 1\end{bmatrix}=\begin{bmatrix}v & 0\\ 0 & w\end{bmatrix}
\]
for some $p_1,p_2,p_3,v,w\in R$.  Comparing the upper right corners of both sides, we have $bs+(a+bsp_1)p_2=0$.  With this equality in hand, then comparing the upper left corners we have $a+bsp_1=v\in \U(R)$.  Hence, we may take $x=sp_1$.

Second, suppose that
\[
Y\begin{bmatrix}1 & p_1\\ 0 & 1\end{bmatrix}\begin{bmatrix}1 &0\\ p_2 & 1\end{bmatrix}\begin{bmatrix}1 & p_3\\ 0 & 1\end{bmatrix}=\begin{bmatrix}v & 0\\ 0 & w\end{bmatrix}
\]
for some $p_1,p_2,p_3,v,w\in R$.  Taking inverses of both sides, then the left-right dual of the computation from the previous paragraph shows that $(2r-rar)-p_1bs=v^{-1}\in \U(R)$.  So,
\[
v^{-1}=r+r(1-ar)-p_1bs=r+(r-p_1)bs.
\]
Applying the left-right symmetric version of \cite[(1.25)]{LamExercises} to the equality $ar+(bs)=1$, then there is some $z\in R$ with $a+bsz\in \U(R)$.  Hence, we may take $x=sz$.
\end{proof}

Over a nonzero ring, the matrix $\left[\begin{smallmatrix}0 & 1\\ 1 & 0\end{smallmatrix}\right]$ cannot be diagonally reduced using fewer than three elementary operations.  So, this provides a natural lower bound on how many elementary operations are generally needed.  Of course, individual matrices may require fewer than three elementary operations to diagonally reduce.

There are separative regular rings that do not have stable rank $1$, such as the ring of column-finite matrices over a field.  So, we have shown that for separative exchange rings, at least four elementary operations are necessary, in general, to diagonally reduce $2\times 2$ invertible matrices.  Our next goal is to show that this number of operations is optimal.  We begin with the following powerhouse lemma about exchange rings.

\begin{lemma}\label{Lemma:NewFactFullIdemExchange}
Let $R$ be an exchange ring, and let $a,b\in R$.  If $aR+bR=R$, then there exists an idempotent $e\in aR$ with $1-e\in bR$.  Moreover, given any idempotent $f\in aR$, we can choose $e$ \textup{(}depending on $f$\textup{)} so that $f\in eR$.  In particular, if $aR\subseteq^{\oplus} R_R$, then we can choose $e$ so that $eR=aR$.
\end{lemma}
\begin{proof}
Assume $f\in aR$ is an idempotent.  Slightly modifying the terminology of \cite[p.\ 440]{Stock}, let us say that a module $P$ has property $(N)_{\rm finite}$ when given any module decomposition $P=\sum_{i\in I}N_i$ with $|I|<\infty$, then there exist submodules $N_i'\subseteq N_i$ with $P=\bigoplus_{i\in I}N_i'$.  From \cite[Proposition 2.9]{Nicholson}, we know that a projective module has the finite exchange property if and only if it has property $(N)_{\rm finite}$.  In particular, $R_R$ has this property.

Note that \cite[Lemma 2.5]{Stock} is still true with $(N)_{\rm finite}$ in place of $(N)$ after limiting the index set $I$ to be finite, using the same proof.  Consequently, since $R=fR\oplus (1-f)R=aR+bR$, then there exist right ideals $A'\subseteq aR$ and $B'\subseteq bR$ such that $R=fR\oplus A'\oplus B'$.

Letting $e\in R$ be the idempotent with $eR=fR\oplus A'$ and $(1-e)R=B'$, then $e\in aR$ and $1-e\in bR$.  Clearly, $f\in eR$.

To prove the last sentence, suppose that $aR\subseteq^{\oplus}R_R$.  Then there is some idempotent $f\in aR$ such that $aR=fR$.  In this case $A'=0$, and so $eR=fR\oplus A'=fR=aR$.
\end{proof}

Given any matrix $A:=\left[\begin{smallmatrix}a & b\\ c & d\end{smallmatrix}\right]\in {\rm GL}_2(R)$, the top row $[a,b]$ is \emph{right unimodular}, meaning
\[
aR+bR=R.
\]
In the opposite direction, if a right unimodular row is the top row of an invertible matrix, we say that it is \emph{completable}.  It is easy to find examples of right unimodular rows that are not completable.  For instance, in the ring ${\rm CFM}_{\omega}(F)$ of $\omega\times \omega$ column-finite matrices over a field $F$, the right unimodular row generated by
\begin{equation}\label{Eq:NotCompletable}
\begin{bmatrix}1 & \phantom{\ddots} & \phantom{\ddots} & \phantom{\ddots} & \phantom{\ddots}\\
0 & 0 & \phantom{\ddots} & \phantom{\ddots} & \phantom{\ddots}\\
\phantom{\ddots}& 1 & 0 & \phantom{\ddots} & \phantom{\ddots}\\
 \phantom{\ddots}& 0 & 0 & 0 & \phantom{\ddots}\\
\phantom{\ddots}& \phantom{\ddots}& 1 & 0 & 0\\
 \phantom{\ddots}& \phantom{\ddots} & 0 & 0 & 0 & 0 & \phantom{\ddots}\\
 \phantom{\ddots}& \phantom{\ddots}& \phantom{\ddots}& \phantom{\ddots}& \phantom{\ddots}& \phantom{\ddots} & \ddots \end{bmatrix}\ \text{ and }\
\begin{bmatrix}0 & \phantom{\ddots} & \phantom{\ddots} & \phantom{\ddots} & \phantom{\ddots}\\
1 & 0 & \phantom{\ddots} & \phantom{\ddots} & \phantom{\ddots}\\
\phantom{\ddots}& 0 & 0 & \phantom{\ddots} & \phantom{\ddots}\\
 \phantom{\ddots}& 1 & 0 & 0 & \phantom{\ddots}\\
\phantom{\ddots}& \phantom{\ddots}& 0 & 0 & 0\\
 \phantom{\ddots}& \phantom{\ddots} & 1 & 0 & 0 & 0 & \phantom{\ddots}\\
 \phantom{\ddots}& \phantom{\ddots}& \phantom{\ddots}& \phantom{\ddots}& \phantom{\ddots}& \phantom{\ddots} & \ddots \end{bmatrix}
\end{equation}
is not completable.  Before we give the proof, let us recall some facts about completability.

Given $m,n\in \Z_{>0}$, identify $R^m$ with the set of column matrices of size $m\times 1$, and do the same for $R^n$.  Then each right $R$-module homomorphism $R^n\to R^m$ corresponds to left multiplication by an $m\times n$ matrix.  In particular, by taking $m=1$, we have a correspondence between rows of length $n$ and their left multiplication action $R^n\to R$.  Clearly, any such row is right unimodular if and only if the corresponding left multiplication map is surjective.

In that case, from projectivity write $R^n\cong P\oplus R$, where $P$ is the kernel of the left multiplication action.  The given right unimodular row is completable if and only if $P\cong R^{n-1}$; see \cite[Corollary 4.9]{LamSerreBook} for the proof.  Specializing to the case when $n=2$, we have:

\begin{cor}
Let $R$ be a ring.  Every $1\times 2$ right unimodular row is completable if and only if $R$ is self-cancellative \textup{(}meaning that for any $P\in \FP(R)$, then $R^2\cong P\oplus R$ implies $P\cong R$\textup{)}.
\end{cor}

Taking $R={\rm CFM}_{\omega}(F)$, then there is an obvious isomorphism $R^2\to R$ given by sending the rows of the first coordinate of $R^2$ to the odd indexed rows in the output and similarly sending the rows of the second coordinate of $R^2$ to the even indexed rows in the output.  This isomorphism is exactly left multiplication by the row generated by \eqref{Eq:NotCompletable}, which explains how we found that row in the first place.  Lack of completability boils down to the fact that the kernel of the homomorphism is zero, and hence it is not isomorphic to $R$.

The next proposition gives us a strong way to complete matrices over exchange rings.

\begin{prop}\label{Prop:CompletableRows}
Let $R$ be an exchange ring, and let $[a,b]\in R^2$ be a completable row.  Then there exist $c,d\in R$ such that $A:=\left[\begin{smallmatrix}a & b\\ c & d\end{smallmatrix}\right]\in {\rm GL}_2(R)$, with $RcR=R$ and $c$ regular.
\end{prop}
\begin{proof}
As $[a,b]$ is completable, fix some matrix $A:=\left[\begin{smallmatrix}a & b\\ \ast & \ast\end{smallmatrix}\right]\in {\rm GL}_2(R)$.  Moreover, as $[a,b]$ is right unimodular, then by Lemma \ref{Lemma:NewFactFullIdemExchange} we can find an idempotent $e\in aR$ with $1-e\in bR$.  Write $e=ar$ and $1-e=bs$ for some $r,s\in R$.

Let $E_1:=\left[\begin{smallmatrix}1 & -rb\\ 0& 1\end{smallmatrix}\right]$.  We find
\[
AE_1 = \begin{bmatrix}a & (1-e)b\\ \ast & z\end{bmatrix},
\]
for some $z\in R$.  The upper-right corner of $AE_1$ is $(1-e)b=bsb$, which is regular with inner inverse $s$.  In particular $R(1-e)b\subseteq^{\oplus}\! _RR$.

Notice that $R(1-e)b+Rz=R$, since $AE_1$ is (left-)invertible.  Using Lemma \ref{Lemma:NewFactFullIdemExchange} a second time, fix an idempotent $f\in R(1-e)b$ with $1-f\in Rz$ and with $R(1-e)b= Rf$.  In particular,
\begin{equation}\label{Eq:NeededOrtho}
(1-e)b(1-f)=0.
\end{equation}
Write $f=t(1-e)b$ and $1-f=uz$, for some $t,u\in R$.

Let $E_2:=\left[\begin{smallmatrix}1 & 0\\ -zt & 1\end{smallmatrix}\right]$.  We find
\[
E_2AE_1 = \begin{bmatrix}a & (1-e)b\\ c_1 & z(1-f)\end{bmatrix},
\]
for some $c_1\in R$.  The lower-right corner of $E_2AE_1$ is $z(1-f)=zuz$, which is regular with inner inverse $u$.  By using Lemma \ref{Lemma:NewFactFullIdemExchange} a third time, we can fix an idempotent $g\in z(1-f)R$ so that $1-g\in c_1R$ and $gR=z(1-f)R$.  Write $g=z(1-f)v$ and $1-g=c_1w$, for some $v,w\in R$.

Let $E_3:=\left[\begin{smallmatrix}1 & 0\\ -(1-f)vc_1+(1-f)v & 1\end{smallmatrix}\right]$.  From \eqref{Eq:NeededOrtho}, the top rows of $E_2AE_1$ and $E_2AE_1E_3$ agree, and we find
\[
E_2AE_1E_3=\begin{bmatrix}a & (1-e)b\\ (1-g)c_1+g & z(1-f)\end{bmatrix}.
\]
Let $c_2:=(1-g)c_1+g$ be the lower left corner entry of $E_2AE_1E_3$.  Since $gc_2=g$, and $(1-g)c_2w=(1-g)$, we see that $Rc_2R=R$.  However, $c_2$ may not yet be regular.  From \cite[Corollary 2.2]{AGOR}, we may fix an idempotent $p\in c_2R$ with $RpR=R$.  Now, by a fourth use of Lemma \ref{Lemma:NewFactFullIdemExchange}, fix an idempotent $h\in c_2R$ with $1-h\in z(1-f)R$ and with $p\in hR$.  Write $h=c_2x$ and $1-h=z(1-f)y$, for some $x,y\in R$.

Let $E_4:=\left[\begin{smallmatrix}1 & 0\\ -(1-f)yc_2 & 1\end{smallmatrix}\right]$.  We find
\[
E_2AE_1E_3E_4=\begin{bmatrix}a & (1-e)b\\ hc_2 & z(1-f)\end{bmatrix}.
\]
Letting $c:=hc_2=c_2xc_2$, we see that $c$ is regular, with inner inverse $x$.  Moreover,
\[
RcR\supseteq RcxR=RhR\supseteq RpR=R.
\]
Finally, multiply on the right by $E_1^{-1}$ to replace $(1-e)b$ with $b$ but leave $a$ and $c$ unchanged.  We see that
\[
E_2AE_1E_3E_4E_1^{-1}=\begin{bmatrix}a & b\\ c & \ast\end{bmatrix}\in {\rm GL}_2(R),
\]
which is a matrix of the form we want.
\end{proof}

Next, we will show that over separative exchange rings, matrices of the type given in Proposition \ref{Prop:CompletableRows} can be diagonally reduced using four elementary operations.

\begin{prop}\label{Prop:FourForSpecial}
Let $R$ be a separative exchange ring.  Assume that $a,b,c,d\in R$ are such that $A:=\left[\begin{smallmatrix}a & b\\ c & d\end{smallmatrix}\right]\in {\rm GL}_2(R)$, with $RcR=R$ and $c$ regular.  Then there exist upper-triangular elementary matrices $U_1,U_2$ and lower-triangular elementary matrices $L_1,L_2$ such that $U_2L_1AU_1L_2$ is diagonal.
\end{prop}
\begin{proof}
We will break the proof into four successive steps.  In the first step we will construct an upper-triangular matrix $U_1$ so that $AU_1=\left[\begin{smallmatrix}a' & b'\\ c' & d'\end{smallmatrix}\right]$ has new special properties.  Throughout this proof primes are never inner inverse operations.
\vspace{\baselineskip}

\noindent{\bf Step 1}: We force $d'\in Rd$, with $d'R=pR$ for some idempotent $p\in R$ that additionally satisfies $R(1-p)R=R$; moreover, we force $Rb'R=R$ with $b'$ regular.

We break this step into three stages.  First, since $cR\subseteq^{\oplus}R_R$, then by Lemma \ref{Lemma:NewFactFullIdemExchange} there exists an idempotent $e\in cR$ with $1-e\in dR$ and with $cR=eR$.  Write $e=cr$ and $1-e=ds$, for some $r,s\in R$.  Note that $r$ is an inner inverse for $c$.

Setting $U_{1,1}:=\left[\begin{smallmatrix}1 & -rd\\ 0 & 1\end{smallmatrix}\right]$, then we have
\[
AU_{1,1}=\begin{bmatrix}
a & b-ard\\ c & (1-e)d
\end{bmatrix}.
\]
Notice that $d':=(1-e)d=dsd\in Rd$ is regular, with inner inverse $s$.  Taking $p:=ds=1-e$, we have $d'R=pR$.  Moreover,
\[
R(1-p)R=ReR=RcR=R.
\]

For the second and third stages we will not change the bottom row of $AU_{1,1}$, and thus not undo any of these facts about $d'$.  Before starting the second stage, we will perform an auxiliary calculation.  (Many of the ideas we will use already appear in the proof of Proposition \ref{Prop:CompletableRows}.)  Notice that $Ra+Rc=R$, and so by Lemma \ref{Lemma:NewFactFullIdemExchange} we may fix an idempotent $f\in Rc$ with $1-f\in Ra$ and with $Rc= Rf$. In particular,
\begin{equation}\label{Eq:SecondOrtho}
c(1-f)=0.
\end{equation}
Write $f=tc$, for some $t\in R$.  Letting $V:=\left[\begin{smallmatrix}1 & -ate\\ 0 & 1 \end{smallmatrix}\right]$, then the top row of $VAU_{1,1}$ is
\[
[a(1-f),b-ard].
\]
In particular, this row is right unimodular.

We are now ready for the second stage.  Using the exchange property, fix an idempotent $g\in a(1-f)R$ with $1-g\in (b-ard)R$.  Write $g=a(1-f)v$, for some $v\in R$.

Let $U_{1,2}:=E_{1,1}+E_{2,2} + [-(1-f)v(b-ard)+(1-f)v]E_{1,2}$.  By \eqref{Eq:SecondOrtho} we see that the bottom rows of $AU_{1,1}$ and $AU_{1,1}U_{1,2}$ are the same, and we find
\[
AU_{1,1}U_{1,2}=\begin{bmatrix}a & g+(1-g)(b-ard)\\c & (1-e)d\end{bmatrix}.
\]
Let $z\in R$ denote the upper-right entry of $AU_{1,1}U_{1,2}$.  Then $z$ is full (i.e., the two-sided ideal it generates is the entire ring) because $g=gz\in RzR$ and also $1-g\in (1-g)(b-ard)R\subseteq RzR$.  The matrix $U_{1,2}$ plays exactly the same role that $E_3$ did in the proof of Proposition \ref{Prop:CompletableRows}.  Analogously, we can construct another upper-triangular elementary matrix $U_{1,3}$, playing a role similar to that of $E_4$ from that other proof, such that right multiplication by $U_{1,3}$ again leaves the bottom row of $AU_{1,1}U_{1,2}$ alone, but replaces the upper-right entry with a full element that is also regular.  The reader is invited to work out the details.

Letting $U_1:=U_{1,1}U_{1,2}U_{1,3}$, we see that $U_1$ is an upper-triangular elementary matrix.  The matrix $AU_1=\left[\begin{smallmatrix}a' & b'\\ c' & d'\end{smallmatrix}\right]$ now has the claimed properties for Step 1.  In Step 2 we will construct a lower-triangular elementary matrix $L_{1}$ such that $L_{1}AU_1=\left[\begin{smallmatrix}a'' & b''\\ c'' & d''\end{smallmatrix}\right]$ satisfies a new special property.
\vspace{\baselineskip}

\noindent{\bf Step 2}:  We force $d''$ to be a unit.

We break this step into two stages.  First, by exactly the same idea as in the first stage of the first step, we can find a lower-triangular matrix $L_{1,1}$ so that the lower-right corner of $L_{1,1}AU_1$, which we can call $d_1$, satisfies $d_1\in d'R$, and $Rd_1=Rq$ for some idempotent $q\in R$ with $R(1-q)R=R$.  Since $d_1$ is regular, we have $d_1R=p'R$ for some idempotent $p'\in R$.  But since
\[
p'R=d_1R\subseteq d'R=pR
\]
we have $R(1-p')\supseteq R(1-p)$, and hence $R(1-p')R=R$.  (In other words, the new element $d_1$ continues to have some of those special properties we forced on $d'$ in Step 1.)

By the argument in the second and third paragraphs of the proof of \cite[Theorem 2.8]{AGOR}, which is the only place we use (a special case of) separativity, we see that $d_1$ is unit-regular.  (Another version of this same argument appears as \cite[Theorem 6.2]{AGOP}.)  Write $d_1=uh$ for some unit $u\in \U(R)$ and some idempotent $h\in R$.

For the second stage, using the exchange property fix an idempotent $e'\in Rd_1=Rh$ such that $1-e'\in Rb'$.  Applying another lower-triangular elementary matrix $L_{1,2}$, we can force the lower-right corner to become $d'':=uh+u(1-h)(1-e')$.  Since $d''$ is a unit if and only if $u^{-1}d''$ is a unit, it suffices to see that
\[
h+(1-h)(1-e')=1-(1-h)e'
\]
is a unit.  This is clear because $(1-h)e'$ squares to zero, since $e'\in Rh$.  Taking $L_1:=L_{1,2}L_{1,1}$, then Step 2 is complete.
\vspace{\baselineskip}

\noindent{\bf Steps 3 and 4}: Force the nondiagonal entries to be zero.

Using the unit entry, we can easily accomplish this task with an elementary column operation followed by an elementary row operation.
\end{proof}

In the previous proof, we used two row and two column operations.  Rewriting the resulting expression, we have
\begin{equation}\label{Eq:SixLower}
A=\begin{bmatrix}1 & 0\\ q_1 & 1\end{bmatrix}\begin{bmatrix}1 & q_2\\ 0 & 1\end{bmatrix}\begin{bmatrix}d_1 & 0\\ 0 & d_2\end{bmatrix}\begin{bmatrix}1 & 0\\ q_3 & 1\end{bmatrix}\begin{bmatrix}1 & q_4\\ 0 & 1\end{bmatrix}
\end{equation}
for some elements $q_1,q_2,q_3,q_4\in R$.  The placement of the diagonal matrix is irrelevant, because we have the normalizing equality \eqref{Eq:NormalizingEquality}, as well as its dual for lower-triangular matrices.  Thus, in hindsight, we could have proved Proposition \ref{Prop:FourForSpecial} using only row operations, or using only column operations.  Exploiting this fact, let us now prove:

\begin{thm}\label{Thm:FourIsEnough}
Let $R$ be a separative exchange ring.  Any element of ${\rm GL}_2(R)$ can be diagonally reduced using four elementary operations.
\end{thm}
\begin{proof}
Given $A\in {\rm GL}_2(R)$, let $[a,b]$ be the top row of $A$.  By Proposition \ref{Prop:CompletableRows}, it is also the top row of another matrix $A'\in {\rm GL}_2(R)$, where the bottom-left entry of $A'$ is full and regular.  Thus, putting \eqref{Eq:NormalizingEquality} together with Proposition \ref{Prop:FourForSpecial}, then there exist two lower-triangular elementary matrices $L_1,L_2$ and two upper-triangular elementary matrices $U_1,U_2$, such that $A'U_1L_1U_2L_2$ is diagonal.

This means that $[a,b]U_1L_1U_2L_2=[u,0]$ for some unit $u\in R$.  But then
\[
[a,b]U_1L_1=[u,0]L_2^{-1}U_2^{-1}=[u,\ast].
\]
Hence, $AU_1L_1$ has a unit entry in its upper-left corner.  Thus, we can fully reduce to a diagonal matrix using just two more elementary operations.
\end{proof}

Looking at the proof of Theorem \ref{Thm:FourIsEnough}, we see that in fact any element of ${\rm GL}_2(R)$ can be diagonally reduced using four elementary (column) operations, effected by starting with an upper-triangular matrix.  By symmetry considerations, we could just have easily have done the same thing starting with a lower-triangular matrix.

\section{Diagonal reduction of more matrices}\label{Section:6NowWorks}

It has long been known that every regular matrix $A$ over a separative exchange ring admits a diagonal reduction, by \cite[Theorem 3.1]{AGOPEarly}.  This means that there are invertible matrices $P,Q$ such that $D:=PAQ$ is diagonal.  If $A$ is $2\times 2$, then so are $P$ and $Q$, and from the results of the previous section we know that each can be expressed as products of at most four elementary matrices and a diagonal invertible matrix, the ``diagonal part.''  Thus, since we can incorporate both diagonal parts into $D$, we see that four row operations and four column operations are sufficient to diagonally reduce $A$.  But are they necessary?

It is straightforward to come up with examples showing that we must use at least three (elementary) row operations and at least three (elementary) column operations, as follows.  Let $[a,b]$ be any right unimodular row over a separative regular ring $R$ that is not completable.  (We gave an explicit example of such a row, generated by the matrices in \eqref{Eq:NotCompletable}.)

First, we claim that using only column operations, it is impossible to transform this row to the form $[x,0]$.  Assuming otherwise, then $x$ would be right invertible, say $xy=1$ for some $y\in R$.  Then the row $[x,0]$ is completable to $\left[\begin{smallmatrix}x & 0\\ 1-yx & y\end{smallmatrix}\right]$ with inverse $\left[\begin{smallmatrix}y & 1-yx\\ 0 & x\end{smallmatrix}\right]$.  Hence, the original row $[a,b]$ would have been completable.

Next, let $U$ and $L$ be arbitrary upper-triangular and lower-triangular (respectively) elementary matrices.  Thus, taking $A_1:=\left[\begin{smallmatrix}a & b\\ 0 & 0\end{smallmatrix}\right]$, then the top row of $LUA_1$ cannot be put in the form $[x,0]$ using any number of column operations.  Similarly, taking $A_2:=\left[\begin{smallmatrix}0 & 0\\ a & b\end{smallmatrix}\right]$, then $ULA_2$ cannot be made diagonal using any number of column operations.  Combining this information, the matrix
\[
\begin{bmatrix}(a,0) & (b,0)\\ (0,a) & (0,b)\end{bmatrix}\in \M_2(R\times R)
\]
cannot be made diagonal using only two row operations and any number of column operations.  By working over the opposite ring $R^{\rm op}$, we can reverse the roles of the column and row operations.  The matrix
\[
\begin{bmatrix}(a,0,a^{\rm op},0) &(b,0,0,a^{\rm op})\\ (0,a,b^{\rm op},0) & (0,b,0,b^{\rm op}) \end{bmatrix}\in \M_2(R\times R\times R^{\rm op}\times R^{\rm op})
\]
needs at least three column operations and at least three row operations to diagonally reduce.

Our next goal is to show that over separative \emph{regular} rings, we can reduce to a diagonal matrix using only three row operations and three column operations, thus improving on the naive bound of four row and four column operations (mentioned in the first paragraph of this section).  We are unsure whether this result applies to regular $2\times 2$ matrices over separative exchange rings.  As we will focus on regular rings, many of the arguments are simpler than those that appeared in the previous section.  We will make essential use of repeaters.

The following proposition describes a series of elementary operations we can use to first produce a unit-regular entry.  In order to simplify notation, let $\mathscr{U}$ be the subgroup of ${\rm GL}_2$ consisting of upper-triangular elementary matrices, and similarly let $\mathscr{L}$ be the subgroup of lower-triangular elementary matrices.

\begin{prop}\label{Prop:Start7}
Let $R$ be a regular ring, and let $A\in \M_2(R)$ be an arbitrary $2\times 2$ matrix.  Then there is a matrix in $\mathscr{U}\mathscr{L}A\mathscr{U}\mathscr{L}$ whose upper left corner entry, which we will call $z$, is a repeater.  In particular, if $R$ is \textup{(}core\textup{)} separative, then $z$ is unit-regular.
\end{prop}
\begin{proof}
Write $A=\left[\begin{smallmatrix}a & b\\ \ast & \ast\end{smallmatrix}\right]$ for some $a,b\in R$.  Letting $U_1:=\left[\begin{smallmatrix}1 & -a'-a'b\\ 0 & 1\end{smallmatrix}\right]$ and $L_1:=\left[\begin{smallmatrix}1 & 0\\ a & 1\end{smallmatrix} \right]$, then $x:=(1-aa')ba$ is the upper left corner of $AU_1 L_1$.  This is a right repeater.

Repeat the previous work, but now using row operations instead of column operations.  Write $AU_1L_1 = \left[\begin{smallmatrix}x & \ast\\ y & \ast\end{smallmatrix}\right]$, for some $y\in R$.  Letting $L_2:=\left[\begin{smallmatrix}1 & 0\\ -x'-yx' & 1\end{smallmatrix}\right]$ and $U_2=\left[\begin{smallmatrix}1 & x\\ 0 & 1\end{smallmatrix}\right]$, then $z:=xy(1-x'x)$ is the upper left corner of $U_2 L_2 A U_1 L_1$.  This is visually a left repeater.  It also remains an right repeater, by Lemma \ref{Lemma:SpecialRegularRepeaters}.  Thus, it is a repeater.
\end{proof}

Note that the previous proposition and proof work just as well for matrices in $\M_n(R)$ with $n>2$, by restricting attention to the northwest $2\times 2$ corner of such matrices.  The proof also shows that using just column operations, we can force the northwest corner to be a right repeater.

With the existence of a unit-regular entry, it is straightforward to diagonally reduce.  First, there is one other result that we will need, generalizing Lemma \ref{Lemma:NewFactFullIdemExchange}.

\begin{lemma}\label{Lemma:StrongerExchangeFacts}
Let $R$ be an exchange ring, and let $a,b\in R$.  If $aR+bR=pR$ for some idempotent $p\in R$, then there exist orthogonal idempotents $e\in aR$ and $f\in bR$ with $e+f=p$.  Moreover, given any idempotent $q\in aR$, we can choose $e$ \textup{(}depending on $q$\textup{)} so that $q\in eR$.  In particular, if $aR\subseteq^{\oplus} R_R$, then we can choose $e$ so that $eR=aR$.
\end{lemma}
\begin{proof}
We have
\[
aR+bR+(1-p)R =R_R=X\oplus qR\oplus (1-p)R,
\]
where $X$ is any direct summand complement of $qR$ in $pR$.  By \cite[Lemma 2.5]{Stock}, with $(N)_{\rm finite}$ in place of $(N)$, we have
\begin{equation}\label{Eq:StrongExchangeDisplay2}
R_R=qR\oplus (1-p)R\oplus A'\oplus B'\oplus C',
\end{equation}
for some submodules $A'\subseteq aR$ and $B'\subseteq bR$ and $C'\subseteq (1-p)R$.

Clearly $C'=0$, since $(1-p)R$ is a direct sum complement to $C'\subseteq (1-p)R$.  Also, since $qR\oplus A'\oplus B'\subseteq aR+bR=pR$, and $pR$ is a direct summand complement to $(1-p)R$, we must have from \eqref{Eq:StrongExchangeDisplay2} that $qR\oplus A'\oplus B'=pR$.

Write $p=e+f$ for orthogonal idempotents $e,f$ such that $eR= qR\oplus A'$ and $fR= B'$.  This is a decomposition into orthogonal (sub)idempotents, as desired.

The last sentence is proved exactly as before.
\end{proof}

\begin{prop}\label{Prop:End7}
Let $R$ be a regular ring, and let $A:=\left[\begin{smallmatrix}a & b\\ c & d\end{smallmatrix}\right]\in \M_2(R)$, where $a\in R$ is unit-regular.  Then there is an element of $\mathscr{L}\mathscr{U}A\mathscr{L}\mathscr{U}$ that is diagonal.
\end{prop}
\begin{proof}
As $a$ is unit-regular, write $a=eu$ for an idempotent $e\in R$ and a unit $u\in R$.  Write $aR+bR=pR$ for some idempotent $p\in R$.  Applying the strong form of Lemma \ref{Lemma:StrongerExchangeFacts}, we can find a decomposition $p=ar+bs$ into orthogonal subidempotents, for some $r,s\in R$, but also guaranteeing that $arR=aR$.  In other words, $eurR=euR$, so by \cite[Ex.\ 21.4]{LamExercises} we may replace $r$ by a unit.

Letting $L_1:=\left[\begin{smallmatrix}1 & 0\\ sr^{-1} & 1\end{smallmatrix}\right]$, then the top row of $AL_1$ is $[pr^{-1},b]$.  Letting $U_{1,1}:=\left[\begin{smallmatrix}1 & -rb\\ 0 & 1\end{smallmatrix}\right]$, then the upper right entry of $AL_1 U_{1,1}$ is zero.  Letting $D:=\left[\begin{smallmatrix}r & 0\\ 0 & 1\end{smallmatrix}\right]$, we can write
\[
AL_1 U_{1,1}D = \begin{bmatrix}p & 0\\ x & y\end{bmatrix}
\]
for some $x,y\in R$.  Multiplying by $D$ simplifies later computations, and we will explain how to remove it later.

Let $h:=x(1-p)[x(1-p)]'$, which is an idempotent.  Setting $U_{1,2}:=\left[\begin{smallmatrix}1 & -(1-p)[x(1-p)]'y\\ 0 & 1\end{smallmatrix}\right]$, then
\[
AL_1 U_{1,1}DU_{1,2}=\begin{bmatrix}p & 0\\ x & (1-h)y\end{bmatrix}.
\]

Next, fix an idempotent $q$ such that $Rp+Rx(1-p)=Rq$.  From the left-right symmetric version of Lemma \ref{Lemma:StrongerExchangeFacts}, we can fix $v,w\in R$ so that $q=vp+wx(1-p)$ is an orthogonal sum of idempotents where $Rp=Rvp$.  Again applying \cite[Ex.\ 21.4]{LamExercises} we may assume $v$ is a unit.

From orthogonality, we know $vpwx(1-p)=0$.  Since $v$ is a unit, $pwx(1-p)=0$.  In other words, $(1-p)wx(1-p)=wx(1-p)$.  Thus, without losing any generality, we might as well have fixed $w\in (1-p)R$.

Now, letting $U_{2}:=\left[\begin{smallmatrix}1 & v^{-1}wh\\ 0 & 1\end{smallmatrix} \right]$, we find
\[
U_2 A L_1 U_{1,1}DU_{1,2} = \begin{bmatrix}p+v^{-1}whx & 0\\ x & (1-h)y\end{bmatrix}.
\]
Set $z:=p+v^{-1}whx$, which is the upper left corner entry of this matrix.  We compute
\begin{eqnarray*}
vz & = & vp+whx \ = \ vp+whx(1-p)+whxp\\
  & = & vp+wx(1-p)+(1-p)whxp\ =\ (1+(1-p)whxp)q.
\end{eqnarray*}
Thus, $[1+(1-p)whxp]^{-1}vz=q$.  Since
\[
Rx\subseteq Rx(1-p)+Rxp\subseteq Rx(1-p)+Rp=Rq,
\]
we can write $x=tq$ for some $t\in R$. Therefore, by taking
\[
L_2:=\begin{bmatrix}1 & 0\\ -t[1+(1-p)whxp]^{-1}v & 1\end{bmatrix},
\]
then $L_2 U_2 A L_1 U_{1,1}DU_{1,2}$ is diagonal.

Of course, multiplying on the right by $D^{-1}$ keeps it diagonal.  As $U_1:=U_{1,1}DU_{1,2}D^{-1}\in \mathscr{U}$, we are done.
\end{proof}

Putting Proposition \ref{Prop:Start7} together with Proposition \ref{Prop:End7}, and using the second paragraph of this section, we now know:

\begin{thm}\label{Thm:SixWithThreeOnBothSides}
Over a \textup{(}core\textup{)} separative regular ring, every $2\times 2$ matrix can be diagonally reduced using just three row operations and three column operations. Nothing less on either side can work in general.
\end{thm}

More precisely, given any matrix $A\in \M_2(R)$, when $R$ is a core separative regular ring, then
\[
\mathscr{L}\mathscr{U}\mathscr{L}A\mathscr{U}\mathscr{L}\mathscr{U}
\]
contains a diagonal element. By a symmetric construction (or, alternatively, noting that conjugation by $\left[\begin{smallmatrix}0 & 1\\ 1 & 0\end{smallmatrix}\right]$ is an automorphism of $\M_2(R)$), we similarly know that
\[
\mathscr{U}\mathscr{L}\mathscr{U}A\mathscr{L}\mathscr{U}\mathscr{L}
\]
always contains a diagonal element.

Analyzing the proof, we similarly see that for any $1\times 2$ matrix $A$ over a core strongly separative ring, then
\[
A \mathscr{U}\mathscr{L}\mathscr{U}
\]
contains a diagonal row matrix $[\ast\ \ 0]$.

The separativity problem can now be restated as the following question:

\begin{question}
Does there exist a regular ring, where some $2\times 2$ matrix over that ring cannot be diagonally reduced using just three row operations and three column operations?
\end{question}

\section*{Acknowledgements}

We thank Kevin O'Meara for posing problems and contributing to many conversations that led to much of this paper.  The third author thanks the Mathematics Department of the Universitat Auton\'oma de Barcelona for their hospitality during his visit. The first, fourth and fifth authors were partially supported by the Spanish State Research Agency (grant No.\ PID2020-113047GB-I00/AEI/10.13039/501100011033). The first and fifth authors were partially supported by the Comissionat per Universitats i Recerca de la Generalitat de Catalunya (grant No.\ 2017-SGR-1725) and by the Spanish State Research Agency through the Severo Ochoa and María de Maeztu Program for Centers and Units of Excellence in R\&D (CEX2020-001084-M). The fourth author was partially supported by PAI III grant FQM-298 of the Junta de Andaluc\'ia, and by the grant ``Operator Theory: an interdisciplinary approach'', reference ProyExcel 00780, a project financed in the 2021 call for Grants for Excellence Projects, under a competitive bidding regime, aimed at entities qualified as Agents of the Andalusian Knowledge System, in the scope of the Plan Andaluz de Investigaci\'on, Desarrollo e Innovaci\'on (PAIDI 2020), Consejer\'ia de Universidad, Investigaci\'on e Innovaci\'on of the Junta de Andaluc\'ia.  This work was partially supported by a grant from the Simons Foundation (\#963435 to Pace P.\ Nielsen).

\providecommand{\MR}{\relax\ifhmode\unskip\space\fi MR }
\providecommand{\MRhref}[2]{%
  \href{http://www.ams.org/mathscinet-getitem?mr=#1}{#2}
}
\providecommand{\href}[2]{#2}

\end{document}